\documentclass[11pt]{amsart}
\usepackage{amscd,amsxtra,amssymb,mathrsfs, bbm}
    \usepackage{combelow}
\usepackage{stmaryrd,mathtools}

\usepackage{geometry}
        {\begin{quotation}\begin{center}\begin{em}}
        {\par\end{em}\end{center}\end{quotation}}

\newtheorem{theorem}{Theorem}[section]
\newtheorem{corollary}[theorem]{Corollary}
\newtheorem{lemma}[theorem]{Lemma}
\newtheorem{proposition}[theorem]{Proposition}

\theoremstyle{definition}
\newtheorem{definition}[theorem]{Definition}
\newtheorem{remark}[theorem]{Remark}
\newtheorem{example}[theorem]{Example}

\DeclareMathOperator{\add}{\mathsf{add}}
\DeclareMathOperator{\ann}{\mathsf{Ann}}

\renewcommand{\ker}{\mathsf{Ker}}

\newcommand{\can}{\mathsf{can}}

\newcommand{\id}{\mathrm{id}}

\DeclareMathOperator{\Coh}{\mathsf{Coh}}
\DeclareMathOperator{\Qcoh}{\mathsf{QCoh}}

\DeclareMathOperator{\VB}{\mathsf{VB}}

\DeclareMathOperator{\Hom}{\mathsf{Hom}}

\DeclareMathOperator{\Ext}{\mathsf{Ext}}

\DeclareMathOperator{\End}{\mathsf{End}}
\DeclareMathOperator{\Mat}{\mathsf{Mat}}
\DeclareMathOperator{\Quot}{\mathsf{Quot}}
\DeclareMathOperator{\Spec}{\mathsf{Spec}}
\DeclareMathOperator{\Sp}{\mathsf{Sp}}
\DeclareMathOperator{\Max}{\mathsf{Max}}

\DeclareMathOperator{\Ob}{\mathsf{Ob}}

\DeclareMathOperator{\Mod}{\mathsf{Mod}}

\DeclareMathOperator{\idm}{\mathfrak{m}}
\DeclareMathOperator{\idp}{\mathfrak{p}}
\DeclareMathOperator{\idq}{\mathfrak{q}}

\newcommand{\kk}{\mathbbm{k}}

\newcommand{\NN}{\mathbb{N}}

\newcommand{\bbX}{\mathbbm{X}}
\newcommand{\bbY}{\mathbbm{Y}}
\newcommand{\bbU}{\mathbbm{U}}
\newcommand{\bbV}{\mathbbm{V}}
\newcommand{\bbW}{\mathbbm{W}}
\newcommand{\bbE}{\mathbbm{E}}

\newcommand{\XX}{\mathbb{X}}

\newcommand{\DD}{\mathbb{D}}

\newcommand{\Ad}{\mathsf{Ad}}
\newcommand{\xhar}{\xhookrightarrow{\phantom{xxx}}}

\newcommand{\kA}{\mathcal{A}}

\newcommand{\kB}{\mathcal{B}}
\newcommand{\kC}{\mathcal{C}}
\newcommand{\kE}{\mathcal{E}}
\newcommand{\kF}{\mathcal{F}}
\newcommand{\kG}{\mathcal{G}}

\newcommand{\kI}{\mathcal{I}}
\newcommand{\kJ}{\mathcal{J}}
\newcommand{\kO}{\mathcal{O}}
\newcommand{\kL}{\mathcal{L}}
\newcommand{\kP}{\mathcal{P}}
\newcommand{\kQ}{\mathcal{Q}}
\newcommand{\kR}{\mathcal{R}}
\newcommand{\kK}{\mathcal{K}}

\newcommand{\kV}{\mathcal{V}}

\newcommand{\kZ}{\mathcal{Z}}

\newcommand{\gS}{\mathfrak{S}}

\newcommand{\lar}{\longrightarrow}

\def\sA{\mathsf A} 
\def\sB{\mathsf B} 
\def\sC{\mathsf C} 
\def\sD{\mathsf D}

\usepackage{tikz}
\usetikzlibrary{calc, arrows, positioning, shapes, fit, matrix, decorations}
\usetikzlibrary{decorations.shapes, decorations.pathreplacing}

\newcommand\dhrightarrow{%
  \mathrel{\ooalign{$\rightarrow$\cr%
  $\mkern3.5mu\rightarrow$}}
}

\input xy
\xyoption{all}

\title[Morita theory for non--commutative schemes]{Morita theory for non--commutative noetherian schemes}

\author{Igor Burban}
\address{
Universit\"at Paderborn\\
Institut f\"ur Mathematik \\
Warburger Stra\ss{}e 100 \\
33098 Paderborn \\
Germany
}
\email{burban@math.uni-paderborn.de}

\author{Yuriy Drozd}
\address{
 Institute of Mathematics\\
National Academy of Sciences of Ukraine,
Tereschenkivska str. 3,
01004 Kyiv, Ukraine}
\email{drozd@imath.kiev.ua, y.a.drozd@gmail.com}

\begin{document}

\begin{abstract}
In this paper,  we study  equivalences between the categories of quasi--coherent sheaves on non--commutative noetherian schemes. In particular, give   a new proof of C\u{a}ld\u{a}raru's conjecture about Morita equivalences of  Azumaya algebras on noetherian schemes. Moreover, we derive necessary and sufficient condition for two reduced non--commutative curves to be Morita equivalent.
\end{abstract}
\maketitle

\section{Introduction} 

\noindent
A classical results of Gabriel (see 
\cite[Section VI.3]{Gabriel}) states  that the categories of quasi--coherent sheaves $\Qcoh(X)$ and $\Qcoh(Y)$ of two separated noetherian schemes $X$ and $Y$ are equivalent if and only if $X$ and $Y$ are isomorphic. To prove this result (and in particular to show how the scheme $X$ can be reconstructed from the category $\Qcoh(X)$), Gabriel used the full power of methods of homological algebra, developed in his thesis \cite{Gabriel}.

\medskip
\noindent
In this work, we deal with similar types of questions for the so--called non--commutative noetherian schemes. By definition, these are ringed spaces $\bbX  = (X, \kA)$, where $X$ is a separated noetherian scheme and $\kA$ is a sheaf of $\kO_X$--algebras, which is coherent viewed as an $\kO_X$--module. A  basic question arising in this context is to establish when the categories of quasi--coherent sheaves 
$\Qcoh(\bbX)$ and $\Qcoh(\bbY)$ on two such non--commutative noetherian schemes $\bbX$ and $\bbY$ are equivalent. 

\medskip
\noindent
We show first  that from the categorical perspective, $\bbX$ and $\bbY$  can without loss of generality assumed to be central; see Subsection \ref{SS:CentralizingNCNS} for details. Following Gabriel's approach \cite{Gabriel}, based on a detailed  study of indecomposable injective objects of $\Qcoh(\bbX)$, we prove that the central scheme $X$ can be recovered from the category $\Qcoh(\bbX)$; see Theorem \ref{T:MorphismCentralSchemes}. Using this reconstruction result, we prove a Morita theorem in the setting of central non--commutative noetherian schemes;  see Theorem \ref{T:MoritaTheoremNCNS} and the discussion afterwards.

\medskip
\noindent
As a first application of this result, we get a new proof of C\u{a}ld\u{a}raru's conjecture about  Azumaya algebras on noetherian schemes; see \cite[Conjecture 1.3.17]{CaldararuCrelle}. Namely, we show that 
if $\bbX = (X, \kA)$ and $\bbY = (Y, \kB)$ are two non--commutative noetherian schemes, such that $\kA$ and $\kB$ are Azumaya algebras on $X$ and $Y$ respectively, then $\Qcoh(\bbX)$ and $\Qcoh(\bbY)$ are equivalent if and only if 
there exists an isomorphism $Y \stackrel{f}\lar X$ such that
$f^\ast\bigl([\kA]\bigr) = [\kB] \in \mathsf{Br(Y)}$, where $\mathsf{Br(Y)}$  is the Brauer group of the scheme $Y$. This result was already proven by Antieau  
\cite{Antieau} (see also \cite{Perego} and \cite{CanonacoStellari}) by much more complicated methods.

\medskip
\noindent
Our main motivation to develop Morita theory in the setting of non--commutative algebraic geometry comes from the study of reduced non--commutative curves. By definition, these are central non--commutative noetherian schemes $\bbX = (X, \kA)$, for which $X$ is a reduced excellent noetherian scheme of pure dimension one and $\kA$ is a sheaf of $\kO_X$--orders. Our  goal was to derive a manageable criterion
to describe the Morita equivalence class of $\bbX$.  

\medskip
\noindent
From the historical perspective,  the so--called projective  hereditary non--commutative curves, i.e.~those $\bbX = (X, \kA)$, for which $X$  is a projective curve over some field $\kk$ and $\kA$ is a sheaf of hereditary orders, were originally of major interest. For $\bbX = \mathbbm{P}^1$, they appeared (in a different form) in the seminal work of Geigle and Lenzing \cite{GeigleLenzing} on weighted projective lines. Tilting theory on these curves  had a significant  impact on the development of the representation theory of finite dimensional $\kk$--algebras.
For an algebraically closed field $\kk$, projective hereditary non--commutative curves play a central role in the classification of  abelian noetherian $\kk$--linear $\Ext$--finite hereditary categories with Serre duality due to Reiten and van den Bergh \cite{ReitenvandenBergh} (see also 
\cite{LenzingReiten, Kussin} for the case of arbitrary fields). In the case of a finite field $\kk$, such  non--commutative curves appeared as  a key technical tool in the work of Laumon, Rapoport and Stuhler \cite{LaumonRapoportStuhler} in the framework of the Langlands programme.  The question of a classification of  non--commutative hereditary curves up to Morita equivalence was clarified by Spie\ss{} in  \cite{Spiess}. In this case,  the Morita equivalence class of such a curve $\bbX$ (however, not $\bbX$ itself, viewed  as a ringed space!) is determined by a central simple algebra $\Lambda_{\bbX}$ (which is an analogue of the function field of a commutative curve) and the types of non--regular points of $\bbX$; see Corollary \ref{C:HereditaryMorita} for details. 

\medskip
\noindent
However, the case of non--hereditary orders happened  to be more tricky. It turns out that even the central curve $X$,  the class of the algebra $\Lambda_{\bbX}$ in the Brauer group of the  function field of $X$ and the isomorphism classes of non--regular points of $\bbX$  are not sufficient to recover $\bbX$ (up to Morita equivalence); see Example \ref{E:ExampleNonHered}.  In Theorem \ref{T:MoritaNonCommCurves},  we give  necessary and sufficient  conditions for two reduced non--commutative curves to be Morita equivalent.

\medskip
\noindent
Non--hereditary reduced non--commutative projective curves  naturally arise as categorical resolutions of singularities of usual singular reduced commutative curves; see 
\cite{BurbanDrozdGavranIMRN}. From the point of view of  representation theory of  finite dimensional $\kk$--algebras, 
the so--called tame non--commutative projective nodal curves seem to be of particular importance; see 
\cite{bd, bdnpdalcurves}. Special classes of such curves  appeared in the framework  of the homological mirror symmetry (in a different language and  under the name  stacky chains/cycles of projective lines) in a work of  Lekili and Polishchuk \cite{LekiliPolishchuk}
as holomorphic mirrors of compact oriented surfaces with non--empty boundary; see also \cite[Section 7]{bdnpdalcurves}.   Getting a precise  description of Morita equivalence classes of tame non--commutative nodal curves (see \cite[Section 4]{bdnpdalcurves}) was   another motivation to carry out this work.

\smallskip
\noindent
\emph{Acknowledgement}.   The work of the first--named author was partially supported by the DFG project Bu--1866/4--1. The results of this paper
were mainly obtained during the stay of the second--named author at the University of Paderborn in September 2018 and September 2019. 

\section{Classical Morita theory and the center}

Through the whole paper we always suppose that the considered rings are associative with unit, modules are unital, 
homomorphisms map unit to unit and subrings contain the unit of the ring.

\subsection{Notation for module theory and reminder of the classical Morita theorem}
\smallskip
\noindent
For any ring $A$, we denote by $A^\circ$ the opposite ring, by $Z(A)$ the center of $A$ and 
by $A-\Mod$ (respectively, 
$\Mod-A$) the category of all left (respectively, right) $A$--modules. 

\smallskip
\noindent
For a commutative ring $R$, an   $R$--algebra  is a pair $(A, \imath)$, where 
$A$ is a ring and $R \stackrel{\imath}\lar A$ an injective homomorphism  such that $\imath(R) \subseteq Z(A)$. If $A$ is a finitely generated $R$--module then one says that $A$ is a \emph{finite} $R$--algebra.
Next,  $(A, \imath)$ is a \emph{central} $R$--algebra if   $\imath(R) = Z(A)$.  Usually, $R$ will be  viewed as a subset of $A$; in this case, the canonical inclusion map $\imath$ will be  suppressed from the notation. We denote by  $A^e := A \otimes_R A^\circ$ the enveloping $R$--algebra of $A$ and   identify the category of $(A-A)$--bimodules with the category  $A^e-\mathsf{Mod}$. The following result is well--known:

\begin{lemma}\label{L:CenterofRing}
If $A$ is an $R$--algebra, then the canonical map $Z(A) \lar \End_{A^e}(A)$ is an isomorphism.
Hence, if   $R$ is noetherian and $A$ is a finite $R$--algebra, then 
\begin{itemize}
\item for any multiplicative subset  $\Sigma \subset R$ we have: 
$\Sigma^{-1}\bigl(Z(A)\bigr) \cong Z\bigl(\Sigma^{-1} A\bigr)$;
\item for any $\idm \in \Max(R)$ we have: $\widehat{Z(A)}_{\idm} \cong
Z\bigl(\widehat{A}_{\idm}\bigr)$. 
\end{itemize}
\end{lemma}
\smallskip
\noindent
Let $A, B$ be any rings and  $P = {}_{B}P_{A}$ be a $(B-A)$--bimodule. Recall that $P$ is called \emph{balanced}, if both structure maps
$$
B \stackrel{\lambda^P}\lar \End_A(P_A), \, b \mapsto \lambda_b^P  \;\; \mbox{\rm and} \; \; 
A^\circ \stackrel{\rho^P}\lar \End_B({}_{B}P), \, a \mapsto \rho_a^P
$$ 
are ring isomorphisms, where  $\lambda_b^P(x) = b x$ and $\rho_a^P(x) = 
x a$ for any $x \in P$, $a \in A$, $b \in B$.

\smallskip
\noindent
For an additive category $\sC$ and $X \in \Ob(\sC)$, we denote by $\add(X)$ the full subcategory of $\sC$, whose objects are direct summands of finite coproducts of $X$. 

\smallskip
\noindent
Let $A$ be any ring and $P$ be a finitely generated right $A$--module. Then $P$ is a \emph{progenerator} of $\Mod-A$ (or just right \emph{$A$--progenerator}) if 
$\add(P) = \add(A)$. In this case, for any
$M \in \Mod-A$ there exists a set $I$ and an epimorphism $P^{\oplus (I)} \lar M$. Other characterizations of right progenerators can be for instance found in \cite[Section 18B]{Lam}.

\smallskip
\noindent
Note that for any $(B-A)$--bimodules ${}_{B} P_{A}$ and ${}_{B} Q_{A}$, the canonical map 
\begin{equation}\label{E:MoritaYoneda}
\Hom_{B-A}\bigl({}_{B} P_{A}, {}_{B} Q_{A}\bigr) \lar 
\Hom\bigl({}_{B} P_{A} \otimes_A -, {}_{B} Q_{A} \otimes_A -\bigr)
\end{equation}
is an isomorphism, where $\Hom$ in the right hand side of (\ref{E:MoritaYoneda}) denotes the abelian group of natural transformations between the corresponding additive functors.

\begin{theorem}[Morita theorem for rings]\label{T:MoritaClassical}
Let $A, B$ be any rings and $$A-\Mod \stackrel{\Phi}\lar B-\Mod$$ be an equivalence of categories. Then we have: $\Phi \cong {}_{B} P_{A} \otimes_{A} -$, where $P$ is a balanced $(B-A)$--bimodule, which is a right progenerator of $A$ (in what follows, such  bimodule will be  called $(B-A)$--Morita bimodule). Moreover, if ${}_{B} Q_{A}$ is another $(B-A)$--Morita bimodule representing $\Phi$ then $P$ and 
$Q$ are canonically isomorphic as bimodules. 
\end{theorem}

\smallskip
\noindent 
A proof of this standard result can be for instance found in \cite[Chapter 18]{Lam}. \qed

\medskip
\noindent
The goal of this work is to generalize Theorem \ref{T:MoritaClassical} to various settings  of non--commutative noetherian schemes.

\subsection{Non--commutative noetherian schemes}

\begin{definition}\label{D:NCNS}
A non--commutative noetherian scheme (abbreviated as \emph{ncns}) is a ringed space $\bbX = (X, \kA)$,
where $X$ is a commutative \emph{separated noetherian} scheme and $\kA$ is a sheaf of $\kO$--algebras \emph{coherent} as $\kO$--module (here, $\kO = \kO_X$ denotes the structure sheaf of $X$).
We say that $\XX$ is \emph{central} if $O_x = Z(A_x)$ for any $x \in X$, where $O_x$ (respectively, $A_x$) is the stalk of $\kO$ (respectively, $\kA$) at the point $x$.
\end{definition}

\smallskip
\noindent
For a ncns $\bbX$, we shall denote by  $\Qcoh(\bbX)$  the category of quasi--coherent sheaves on $\bbX$, i.e. the category of sheaves of left $\kA$--modules which are quasi--coherent as sheaves of $\kO$--modules. 
For an open subset $U \subseteq X$ and $\kF \in \Qcoh(X)$, we shall use both notations
$\Gamma(U, \kF)$ and $\kF(U)$ for the corresponding group of local sections and write
$O(U) = \kO(U)$ and  $A(U) = \kA(U)$. Note that $A(U)$ is a finite $O(U)$--algebra. Moreover,   for any pair of open affine subsets $V \subseteq U \subseteq X$, the canonical map $O(V) \otimes_{O(U)} A(U) \rightarrow A(V)$ is an isomorphism of
$O(V)$--algebras. Similarly, for any $\kF \in \Qcoh(\bbX)$, the canonical map
\begin{equation}
O(V) \otimes_{O(U)} \Gamma(U, \kF) \lar \Gamma(V, \kF)
\end{equation}
is an isomorphism of $A(V)$--modules.

\smallskip
\noindent
For any open subset $U \subseteq X$, we get a ncns $\bbU := (U, \kA\big|_{U})$. Since $X$ is assumed to be noetherian, it admits a \emph{finite} open covering $X = U_1 \cup \dots \cup U_n$, where
$U_i = \Spec(R_i)$ for some  noetherian ring $R_i$. For any $1 \le i \le n$, let $A_i := A(U_i)$. As in
\cite[Chapitre VI]{Gabriel}, one can easily show that $\Qcoh(\bbX)$ is equivalent to an iterated \emph{Gabriel's recollement} of the abelian categories $A_1-\Mod, \dots, A_n-\Mod$ (see also Definition \ref{D:Recollement} below).
Since $\kA$ is a coherent  $\kO$--module, all rings
$A_1, \dots, A_n$ are noetherian. As in \cite[Chapitre VI, Th\'eor\`eme 1]{Gabriel}, one concludes that $\Qcoh(\bbX)$ is a \emph{locally noetherian} abelian category, whose subcategory
of \emph{noetherian objects} is the category $\Coh(\bbX)$ of coherent sheaves on $\bbX$; see 
\cite[Section II.4]{Gabriel} for the corresponding definitions. 

\smallskip
\noindent
Let $\bbX$ and $\bbY$ be two ncns and $\Qcoh(\bbX) \stackrel{\Phi}\lar \Qcoh(\bbY)$ be an equivalence of categories. It is clear that $\Phi$ restricts to an equivalence
$\Coh(\bbX) \stackrel{\Phi_{|}}\lar \Coh(\bbY)$ between the corresponding subcategories of noetherian objects. However, \cite[Section II.4, Th\'eor\`eme 1]{Gabriel} asserts that conversely, any equivalence $\Coh(\bbX) \rightarrow \Coh(\bbY)$ admits a unique (up to an isomorphism of functors) extension to
an equivalence $\Qcoh(\bbX) \rightarrow \Qcoh(\bbY)$ (in  \cite{Gabriel}, this result is attributed to Grothendieck and Serre). Hence, even being primarily interested in the study  of the category $\Coh(\bbX)$, it is technically more advantageous to work with  a larger category 
$\Qcoh(\bbX)$. One of the main reasons for this is a good behavior of the set $\Sp(\bbX)$
of the isomorphism classes of indecomposable injective objects of $\Qcoh(\XX)$ (see \cite[Section  IV.2]{Gabriel}), for which it is crucial  that $\Qcoh(\XX)$ is locally noetherian. 

\smallskip
\noindent
Note that if we assume $X$ to be just \emph{locally noetherian} then even the category $\Qcoh(X)$ need not be locally noetherian in general; see \cite[Section II.7]{ResiduesDuality}. Hence, dropping the assumption for a ncns $\bbX$ to be noetherian would lead to significant technical complications. 

\subsection{Reminder on the center of an additive category}
\begin{definition}
The  \emph{center} $Z(\sA)$ of an additive category $\sA$ is the set of endomorphisms of the identity functor $\mathsf{Id}_{\sA}$, i.e.
\begin{equation*}
Z(\sA) := 
\left\{
\eta = \bigl(\bigl(X \stackrel{\eta_X}\lar X\bigr)_{X \in \mathsf{Ob}(\sA)}\bigr) \left| 
\begin{array}{c}
\xymatrix{
X \ar[r]^-{\eta_X} \ar[d]_-{f} & X \ar[d]^-{f} \\
X' \ar[r]^-{\eta_{X'}} & X'
}
\end{array}
\right.
\;
\mbox{\rm is commutative for all} \; X \stackrel{f}\lar X'
\right\}.
\end{equation*}
It is easy to see that $Z(\sA)$ is a commutative ring. 
\end{definition}

\smallskip
\noindent
It is well--known (see e.g. \cite[Proposition II.2.1]{Bass}) that for   any ring $A$, the map 
\begin{equation}\label{E:CenterCenter}
Z(A) \stackrel{\upsilon}\lar  Z(A-\Mod), \; r \mapsto (\lambda^M_r)_{M \in \Ob(\sA)}
\end{equation}
is a ring isomorphism.

\medskip
\noindent
The following result must be well--known. Its proof reduces to lengthy but completely straightforward verifications and is  therefore left to an interested reader. 
\begin{proposition}\label{P:Centers}
Let $\sA$ and $\sB$ be additive categories, $\eta \in Z(\sA)$ and $\sA \stackrel{\Phi}\lar 
\sB$ be an additive functor satisfying the following conditions:
\begin{itemize}
\item $\Phi$ is essentially surjective. 
\item 
For any $X_1, X_2 \in \Ob(\sA)$ and $g \in \Hom_{\sB}\bigl(\Phi(X_1), \Phi(X_2)\bigr)$, there exist $X \in \Ob(\sA)$ and morphisms $X_1 \stackrel{t}\longleftarrow X \stackrel{f}\longrightarrow X_2$ in $\sA$ such that $\Phi(t)$ is an isomorphism and  $g = \Phi(f) \cdot\bigl(\Phi(t)\bigr)^{-1}$.
\end{itemize}
For any $Y \in \Ob(\sB)$ choose a pair $\bigl(X_Y, \xi_Y\bigr)$, where $X_Y \in \Ob(\sA)$ and
$\Phi(X_Y) \stackrel{\xi_Y}\lar Y$ is an isomorphism. 
Then the following statements are true. 

\smallskip
\noindent
$\bullet$ The unique endomorphism $\vartheta_Y \in \End_{\sB}(Y)$, which makes the diagram 
$$
\xymatrix{
\Phi(X_Y) \ar[r]^-{\xi_Y} \ar[d]_-{\Phi\bigl(\eta_{X_Y} \bigr)} & Y \ar[d]^-{\vartheta_Y}\\
\Phi(X_Y) \ar[r]^-{\xi_Y} & Y 
}
$$
commutative, does not depend on the choice of the pair $\bigl(X_Y, \xi_Y\bigr)$.

\smallskip
\noindent
$\bullet$ Let $\vartheta = \bigl(\vartheta_Y\bigr)_{Y \in \Ob(\sB)}$, then we have: $\vartheta
\in Z(\sB)$ (in other words, for any $\eta \in Z(\sA)$, the family of endomorphisms
$\bigl(\Phi(\eta_X)\bigr)_{X \in \Ob(\sA)}$ in the category $\sB$ can be uniquely extended to an element $\vartheta
 \in Z(\sB)$). Moreover, the map
$
Z(\sA) \stackrel{\Phi_c}\lar Z(\sB), \; \eta \mapsto \vartheta
$
is a ring homomorphism. 

\smallskip
\noindent
$\bullet$  Let  $\sA \stackrel{\Psi}\lar 
\sB$ be   a functor such that  $\Phi \cong \Psi$.  Then  the induced maps of the corresponding centers  are equal: $\Phi_c = \Psi_c$. 
Finally, if  $\sA_1, \sA_2, \sA_3$ are additive categories and $\sA_1 \stackrel{\Phi_1}\lar 
\sA_2 \stackrel{\Phi_2}\lar 
\sA_3$ additive functors, satisfying the conditions of this proposition then we have: $
\bigl(\Phi_2 \Phi_1\bigr)_c = \bigl(\Phi_2)_c  \bigl(\Phi_1)_c.
$
\end{proposition}

\smallskip
\noindent
 From the point of view of applications in this paper, the  following two classes of functors satisfying the conditions of 
Proposition \ref{P:Centers} are of major interest:
\begin{itemize}
\item Equivalences of additive categories.
\item Serre quotient functors $\sA \rightarrow  \sA/\sC$, where $\sC$ is a Serre subcategory of an abelian category $\sA$; see for instance \cite[Section 4.3]{Popescu}.
\end{itemize}

\begin{lemma}\label{L:MoritaCentralMap}
Let $A, B$ be any rings and $P$ be a $(B-A)$--Morita bimodule. Then there exists a unique isomorphism of centers $Z(A) \stackrel{\varphi}\lar Z(B)$ making the diagram 
\begin{equation}\label{E:CommDiag}
\begin{array}{c}
\xymatrix{
\End_A(P) & B \ar[l]_-{\lambda^P} \\
{Z(A)\hbox{\large$\phantom{\arrowvert}$}} \ar[r]^-{\varphi} \ar@{^{(}->}[u]^-{\rho^P} & {Z(B)\hbox{\large$\phantom{\arrowvert}$}} \ar@{_{(}->}[u]
}
\end{array}
\end{equation}
commutative. In other words, for any $a \in Z(A)$ and $x \in P$ we have: 
$\varphi(a) \cdot x = x \cdot a$. Moreover, $\varphi = \Phi_c$, where
$\Phi := P \otimes_A -: \; A-\Mod \lar B-\Mod$.
\end{lemma}

\begin{proof} Since $P$ is a balanced $(B-A)$--bimodule, the map $\lambda^P$ is bijective. This implies the uniqueness of $\varphi$. To show the existence, we prove that the induced map
of centers $Z(A) \stackrel{\Phi_c}\lar Z(B)$  
makes the diagram (\ref{E:CommDiag}) commutative. Let $a \in Z(A)$, $b = \Phi_c(a)$ and $\vartheta = \upsilon(b) \in Z(B-\Mod),$ where $\upsilon$ is the map from (\ref{E:CenterCenter}). Then we have: $\vartheta_P = \lambda^P_b$. Let $P \otimes_A A \stackrel{\gamma}\lar  P$ be the canonical isomorphism, then 
the following diagram 
$$
\xymatrix{
P \otimes_A A \ar[d]_-{\gamma} \ar[rr]^-{\mathsf{id}_P \otimes \lambda_a^A} & & P \otimes_A A \ar[d]^-{\gamma} \\
P \ar[rr]^-{\lambda_b^P} & & P
}
$$
is commutative (see 
Proposition \ref{P:Centers}),  what implies the statement. 
\end{proof}

\begin{remark}\label{R:Pseudocommutative}
Let $A$ and $B$ be two rings, $P$ be a $(B-A)$--Morita bimodule and $\Phi = P \otimes_A -$ be  the corresponding equivalence of categories. We may regard  $\Phi$ as  a ``virtual'' ring homomorphism $\xymatrix{A \ar@{.>}[r]^-{\Phi} & B}$. Then the commutativity of the diagram  (\ref{E:CommDiag}) can be rephrased  by saying that the diagram 
\begin{equation}\label{E:CommDiagr}
\begin{array}{c}
\xymatrix{
A \ar@{.>}[r]^-{\Phi} & B  \\
{Z(A)\hbox{\large$\phantom{\arrowvert}$}} \ar[r]^-{\varphi} \ar@{^{(}->}[u] & {Z(B)\hbox{\large$\phantom{\arrowvert}$}} \ar@{_{(}->}[u]
}
\end{array}
\end{equation}
is ``commutative''. Assume additionally that $A$ and $B$ are central $R$--algebras. 
We call  an equivalence $\Phi = P \otimes_A \,-\,$  \emph{central} if the induced map 
$R \stackrel{\Phi_c}\lar R$ is the identity.  According to Lemma \ref{L:MoritaCentralMap},  $\Phi$ is central if and only if for any $r$ in $R$ and $x \in P$ we have: $r \cdot x = x \cdot r$. 
\end{remark}

\begin{definition}\label{D:Recollement}
Let $\sA, \sB, \sD$ be abelian categories and
$
\sA \stackrel{\Phi}\lar \sD \stackrel{\Psi}\longleftarrow \sB
$
be exact functors.
The \emph{Gabriel's recollement} $\sA {\prod \atop \sD} \sB$
 is the category,  whose objects are triples 
$$
\left\{ 
(X, Y, f) \left| \begin{array}{c} 
X \in \Ob(\sA) \\
 Y \in \Ob(\sB) 
\end{array}
\right.
\Phi(X) \stackrel{f}\lar \Psi(Y) \; \mbox{is an isomorphism in}\; \sD
\right\}
$$
and a morphism $(X, Y, f) \stackrel{(\alpha, \beta)}\lar (X', Y', f')$ 
is given by morphisms $X \stackrel{\alpha}\lar X'$ and $Y \stackrel{\beta}\lar Y'$  such that 
$\Psi(\beta) f = f' \Phi(\alpha)$ see \cite[Section VI.1]{Gabriel}.
\end{definition}

\smallskip
\noindent
It is not difficult to check that 
the category  $\sC = \sA {\prod \atop \sD} \sB$  is abelian.
 Assume additionally, that $\Phi$ and $\Psi$ are localization functors,  i.e.~that they induce equivalences of categories
$$
\sA/\mathsf{Ker}(\Phi) \stackrel{\overline\Phi}\lar \sD \stackrel{\overline\Psi}\longleftarrow \sB/\mathsf{Ker}(\Psi),
$$
and  admit right adjoint functors
$
\sA \stackrel{\widetilde\Phi}\longleftarrow \sD   \stackrel{\widetilde\Psi}\lar \sB
$ (see \cite[Section III.2]{Gabriel}). Then we have a  diagram of abelian categories and functors
\begin{equation}\label{E:Recollement}
\begin{array}{c}
\xymatrix{
\sC \ar[r]^-{\Phi^\dagger} \ar[d]_-{\Psi^\dagger} & \sB \ar[d]^-{\Psi}\\
\sA \ar[r]^-{\Phi}  & \sD
}
\end{array}
\end{equation}
where $\Psi^\dagger(X, Y, f) = X$ and $\Phi^\dagger(X, Y, f) = Y$. Moreover,   $\Phi^\dagger$ and $\Psi^\dagger$ are localization functors and $\Phi \Psi^\dagger \cong 
\Psi \Phi^\dagger$.

\smallskip
\noindent
\begin{lemma}\label{P:CentersRecollement}
In the above setting, let $A, B, C, D$ be the centers of the categories $\sA, \sB, \sC$ and $\sD$, respectively. Then  (\ref{E:Recollement}) induces a commutative diagram in the category of rings
\begin{equation}\label{E:RecollCenters}
\begin{array}{c}
\xymatrix{
C \ar[r]^-{\Phi^\dagger_c} \ar[d]_-{\Psi^\dagger_c} & B \ar[d]^-{\Psi_c}\\
A \ar[r]^-{\Phi_c}  & D,
}
\end{array}
\end{equation}
which is moreover a \emph{pull--back} diagram. In other words, we have: 
$$
C \cong A \times_D B := \bigl\{(a, b) \in A \times B \; \bigl| \; \Phi_c(a) = \Psi_c(b)  \bigr\}.
$$
\end{lemma}

\smallskip
\noindent
\emph{Comment to the proof}. This statement is a consequence of  Proposition \ref{P:Centers}. \qed

\smallskip
\noindent
We conclude this subsection with the following categorical version of the classical Skolem--Noether theorem. 

\begin{theorem}\label{T:SkolemNoether}
Let $\kk$ be a  field, $\Lambda$ and $\Gamma$ two semi--simple finite dimensional $\kk$--algebras and 
$
\xymatrix{
\Lambda-\mathsf{Mod}   \ar@/^2pt/[r]^{\Phi}    \ar@/_2pt/[r]_{\Psi} & \Gamma-\mathsf{Mod}
}
$
two equivalences of categories such that
$\Phi_c = \Psi_c$. Then we have: $\Phi \cong  \Psi$.
\end{theorem}
\begin{proof} Let $K = Z(\Lambda)$, $L = Z(\Gamma)$ and $K \stackrel{\varphi}\lar L$ be
the common isomorphism of centers induced by the equivalences $\Phi$ and $\Psi$ (i.e.
$\Phi_c = \varphi = \Psi_c$). Next, let  $P$ and $Q$ be $(\Gamma-\Lambda)$--bimodules such that $\Phi = P \otimes_\Lambda \,-\,$ and $\Psi = Q \otimes_\Lambda \,-\,$. Let $\gamma = \bigl(\lambda_\Gamma^Q\bigr) \circ \bigl(\lambda_\Gamma^P\bigr)^{-1}$.  By Lemma 
\ref{L:MoritaCentralMap},
the following diagram of $\kk$--algebras and algebra homomorphisms
\begin{equation}\label{E:CDiag}
\begin{array}{c}
\xymatrix{
\Gamma \ar@/_5.5ex/[dd]_-{\id} \ar[rr]^-{\lambda_\Gamma^P} & & \End_\Lambda(P) \ar@/^5.5ex/[dd]^-{\gamma}\\
{\hbox{\large$\phantom{\arrowvert}$\!}L} \ar@{^{(}->}[u] \ar@{_{(}->}[d] & 
& \ar[ll]_-{\varphi} {\hbox{\large$\phantom{\arrowvert}$}K} \ar@{_{(}->}[u]_-{\varrho^P_K} \ar@{^{(}->}[d]^-{\varrho^Q_K} \\
\Gamma \ar[rr]^-{\lambda_\Gamma^Q} & & \End_\Lambda(Q) \\
}
\end{array}
\end{equation}
is commutative. In particular, we have: $\gamma \circ \varrho^P_K = \varrho^Q_K$.

\smallskip
\noindent
Since $\Lambda$ is a semi--simple $\kk$--algebra, there exist simple algebras $\Lambda_1, \dots, \Lambda_t$ such that $\Lambda \cong \Lambda_1 \times \dots \times \Lambda_t$. Moreover, for any $1 \le i \le t$ there exists a finite dimensional skew field $F_i$ over $\kk$ such that $\Lambda_i \cong \mathsf{Mat}_{m_i}(F_i)$ for some $m_i \in \NN$. If $K_i := Z(F_i)$ then we have: $K \cong K_1 \times \dots \times K_t$. Let $U_i$ be a finite dimensional simple right $\Lambda_i$--module (which is unique up to an isomorphism). Then we have: $F_i \cong \End_{\Lambda_i}(U_i)$. Moreover, we have direct sum decompositions
$
P \cong U_1^{\oplus p_1} \oplus \dots \oplus U_t^{\oplus p_t}$ and $
Q \cong U_1^{\oplus q_1} \oplus \dots \oplus U_t^{\oplus q_t}.
$
Then we get:
$$
\End_\Lambda(P) \cong \mathsf{Mat}_{p_1}(F_1) \times \dots \times \mathsf{Mat}_{p_t}(F_t)
\quad \mbox{\rm and} \quad
\End_\Lambda(Q) \cong \mathsf{Mat}_{q_1}(F_1) \times \dots \times \mathsf{Mat}_{q_t}(F_t).
$$
It follows from (\ref{E:CDiag}) that there exists  an isomorphism of \emph{$K$--algebras} (and not just of \emph{$\kk$--algebras})
$
\mathsf{Mat}_{p_1}(F_1) \times \dots \times \mathsf{Mat}_{p_t}(F_t)
\lar \mathsf{Mat}_{q_1}(F_1) \times \dots \times \mathsf{Mat}_{q_t}(F_t),
$
what implies that $p_i = q_i$ for all $1 \le i \le  t$. In particular, $P$ and $Q$ are isomorphic, at least as right $\Lambda$--modules. 

\smallskip
\noindent
Let $P \stackrel{h}\lar Q$ be any isomorphism of right $\Lambda$--modules. Then the following diagram
\begin{equation*}
\begin{array}{c}
\xymatrix{
\End_\Lambda(P) \ar[r]^-{\Ad_h}& \End_\Lambda(Q) \\
{\hbox{\large$\phantom{\arrowvert}$\!}K} \ar@{_{(}->}[u]^-{\rho^P_K} \ar[r]^-{\id} & {\hbox{\large$\phantom{\arrowvert}$\!}K} \ar@{^{(}->}[u]_-{\rho^Q_K}
}
\end{array}
\end{equation*}
is commutative, i.e.~the isomorphism $\Ad_h$ is central.  Indeed, for any $f \in \End_\Lambda(P)$ we have: $\Ad_h(f) h = h f$. For any $\lambda \in K$ consider the endomorphism
$\varrho_\lambda^P \in \End_\Lambda(P)$. Since $h$ is $K$--linear, we have: $h \varrho_\lambda^P = \varrho_\lambda^Q h$. Hence, $\varrho_\lambda^Q = \Ad_h(\varrho_\lambda^P)$ for 
any $\lambda \in K$. 

\smallskip
\noindent
Consider the map $\delta:= \Ad_h \cdot  \gamma^{-1}: \End_\Lambda(Q) \lar \End_\Lambda(Q)$. From what was said above it follows that $\delta$ is an isomorphism of $K$--algebras. Now we can finally apply the classical Skolem--Noether theorem: there exists $\bar{h} \in \End_\Lambda(Q)$ such that $\delta = \Ad_{\bar{h}}$. Consider the isomorphism of right $\Lambda$--modules $g = \bar{h}^{-1} h: P \lar Q$. Then we have: 
$
\gamma = \Ad_{\bar{h}}^{-1} \Ad_h = \Ad_{g}.
$
It follows from commutativity of (\ref{E:CDiag}) that the diagram 
$$
\xymatrix{
& \Gamma \ar[rd]^-{\lambda_\Gamma^Q} \ar[ld]_-{\lambda_\Gamma^P}& \\
\End_\Lambda(P) \ar[rr]^-{\Ad_g} & & \End_\Lambda(Q)
}
$$
is commutative, too. Hence,  $P \stackrel{g}\lar Q$ is also $\Gamma$--linear. Summing up, $g$ is an isomorphism of $(\Gamma-\Lambda)$--bimodules and 
$\Phi \cong \Psi$, as asserted. 
\end{proof}

\subsection{Centralizing a non--commutative 
noetherian schemes}\label{SS:CentralizingNCNS}
The goal of this subsection is to show, that any ncns can be replaced by a Morita equivalent central ncns.

\begin{proposition}\label{P:CentralSheaf} 
Let $\bbX = (X, \kA)$ be a ncns. For all open subsets $U \subseteq X$  we put:
\begin{equation}
\Gamma(U, \kZ_\kA):= 
\Bigl\{\alpha \in \Gamma(U, \kA) \, \left|\, \alpha\big|_{V} \in 
Z\bigl(\Gamma(V, \kA)\bigr) \; \mbox{for all}\; V \subseteq U \;\mbox{open} \right.\Bigr\}.
\end{equation}
Then  $\kZ = \kZ_\kA$ is a coherent sheaf on $X$ such that $
\kZ_x \cong Z(A_x)$ for any $x \in X$. Moreover, the canonical map 
\begin{equation}\label{E:CatCenterNcns}
\Gamma(X, 
\kZ) \stackrel{\upsilon_{\bbX}}\lar 
Z\bigl(\Qcoh(\XX)\bigr)
\end{equation} is a ring  isomorphism.
\end{proposition}

\begin{proof} It is clear that $\kZ_\kA$ is a presheaf of commutative rings, which is a sub--presheaf of $\kA$. We have to check the sheaf property of $\kZ$. Let $U \subseteq X$ be any open subset, $U = \cup_{i \in I} U_i$ an open covering and 
$$
\bigl(\alpha_i \in \Gamma\bigl(U_i, \kZ\bigr)\bigr)_{i \in I} \;
\mbox{\rm be such that}\; \alpha_k\big|_{U_k \cap U_l} =  \alpha_l\big|_{U_k \cap U_l} \; \mbox{\rm for all}\; k, l \in I.
$$
Then there exists a unique section $\alpha \in \Gamma(U, \kA)$ such that 
$\alpha\big|_{U_i} = \alpha_i$ for all $i \in I$. We have to show that
$\alpha\big|_{V} \in Z\bigl(\Gamma(V, \kA)\bigr)$ for any open subset $V \subseteq U$. Consider any  $\beta \in \Gamma(V, \kA)$. We have to prove that
$\bigl[\alpha\big|_{V}, \beta\bigr] = 0$. Indeed, we have:
$V = \cup_{i \in I} V_i$, where $V_i = V \cap U_i$. Since 
$\alpha\big|_{V_i} = \alpha_i\big|_{V_i}$ and $\alpha_i \in \Gamma\bigl(U_i, \kZ_\kA\bigr)$, we conclude that $\alpha\big|_{V_i} \in Z\bigl(\Gamma(V_i, \kA)\bigr)$. It implies that
$
\bigl[\alpha\big|_V, \beta\bigr]\Big|_{V_i} = 0
$
for all $i \in I$.  Hence, $\bigl[\alpha\big|_V, \beta\bigr] = 0$.

\smallskip
\noindent
Let  $\alpha \in \Gamma(X, \kZ)$. Then for any 
$\kF \in \Ob\bigl(\Qcoh(\bbX)\bigr)$, we have an endomorphism $\alpha_\kF \in 
\End_{\bbX}(\kF)$ given for any open subset $U \subseteq X$ by the rule
$$
\Gamma(U, \kF) \stackrel{\alpha^{U}_\kF}\lar \Gamma(U, \kF), \; f \mapsto \alpha\big|_{U} \cdot \varphi.
$$
It is clear that the collection of endomorphisms $(\alpha_\kF)$ defines an element
of $Z\bigl(\Qcoh(\XX)\bigr)$, which we denote by $\upsilon_{\bbX}(\alpha)$ (it is how the canonical map $\upsilon_{\bbX}$ from (\ref{E:CatCenterNcns}) is actually defined). If $\alpha \ne 0$ then $\alpha_\kA \ne 0$, too. Hence, the map $\upsilon_{\bbX}$ is at least injective.

\smallskip
\noindent
To show the surjectivity of $\upsilon_{\bbX}$, assume first that $X$ is affine. Let  $A := \Gamma(X, \kA)$, then the functor of global sections 
$
\Qcoh(\bbX) \stackrel{\Gamma}\lar A-\Mod
$
is an equivalence of categories and the induced map of centers 
$Z(A) \lar Z\bigl(\Qcoh(\bbX)\bigr)$ is an isomorphism. 
In the same way as above one can show that 
 $\alpha\big|_{V} \in 
Z\bigl(\Gamma(V, \kA)\bigr)$ for any open subset $V \subseteq X$ and $\alpha 
\in Z(A)$.

\smallskip
\noindent
Next, note  that we have a  sheaf isomorphism $\kZ\big|_{U}
\cong \kZ_{\kA|_{U}}$ for any open subset $U \subseteq X$. If $U$ is moreover affine, 
it follows   that $\Gamma(U, \kZ) = 
Z\bigl(\Gamma(U, \kA)\bigr)$.

\smallskip
\noindent
Now, we  prove by induction on the minimal  number of affine open charts of an affine open covering of $X$  that $\upsilon_{\bbX}$ is an isomorphism. The case  of an affine scheme $X$ is already established. Assume that this statement is true  for any nccs, which can be covered by $n$ affine charts. Suppose that we have an affine open covering 
$X = U_1 \cup \dots \cup U_n \cup U_{n+1}$. Let $U = U_1 \cup \dots \cup U_n$ and 
$V = U_{n+1}$. Since $X$ is separated,  $(U_1 \cap V) \cup \dots \cup (U_n \cap V)$ is  an affine open covering of $W:= U \cap V$  and the category 
$\Qcoh(\bbX)$ is equivalent to Gabriel's recollement with respect to the diagram
$$
\Qcoh(\bbU) \stackrel{\left(\varrho^{U}_W\right)^\ast}\lar \Qcoh(\bbW) 
\stackrel{\left(\varrho^{V}_W\right)^\ast}\longleftarrow \Qcoh(\bbV).
$$
Let $R := Z\bigl(\Qcoh(\bbU)\bigr)$, $S := Z\bigl(\Qcoh(\bbV)\bigr)$,
$C := Z\bigl(\Qcoh(\bbX)\bigr)$ and $T = Z\bigl(\Qcoh(\bbW)\bigr)$. Then Lemma 
\ref{P:CentersRecollement} implies that 
$C \cong R \times_T S$.

\smallskip
\noindent
On the other hand, we have a commutative diagram of rings and ring homomorphisms
$$
\xymatrix{
\Gamma(U, \kZ)  \ar[r] \ar[d]_-{\upsilon_{\bbU}} & \Gamma(W, \kZ) 
\ar[d]_-{\upsilon_{\bbW}} & \ar[l] \Gamma(V, \kZ) \ar[d]^-{\upsilon_{\bbV}} \\
R \ar[r] & T & \ar[l] S
}
$$
in which all vertical maps are isomorphisms due  to the hypothesis of induction. 
The sheaf property of $\kZ$ implies that the map $\upsilon_{\bbX}$ is surjective, hence bijective. 

\smallskip
\noindent
The fact that the sheaf $\kZ$ is coherent and  that we have isomorphism $
\kZ_x \cong Z(A_x)
$ for any $x \in X$ are now easy consequences of Lemma \ref{L:CenterofRing}.
\end{proof}

\begin{corollary}\label{C:Centers}
Let $\bbX = (X, \kA)$ be a ncns and $V \subseteq U \subseteq X$ be open subsets. Then the following diagram of rings and ring homomorphisms 
\begin{equation}\label{E:ReconstrCenter}
\begin{array}{c}
\xymatrix{
\Gamma(U, \kZ) \ar[rr]^{\upsilon_\bbU} \ar[d]_-{\varrho^U_V}& & 
Z\bigl(\Qcoh(\bbU)\bigr) \ar[d]^-{\varphi^U_V}\\
\Gamma(V, \kZ) \ar[rr]^{\upsilon_\bbV} & & Z\bigl(\Qcoh(\bbV)\bigr)
}
\end{array}
\end{equation}
is commutative, where $\varphi^U_V$ is the morphism of centers induced by the localization
functor $\Qcoh(\bbU) \rightarrow  \Qcoh(\bbV)$. Moreover, the horizontal maps in (\ref{E:ReconstrCenter}) are isomorphisms.
\end{corollary}

\begin{remark}\label{R:CenterSheaf} Let $\bbX = (X, \kA)$ be a ncns.  For any open subset 
$V \subseteq X$, consider the  map
$$
\Gamma(V, \kZ) \lar 
\End_{A(V)^e}\bigl(A(V)\bigr), \; \alpha \mapsto \bigl(\beta \mapsto \alpha \cdot \beta \bigr).
$$
These maps  define a morphism of sheaves of $\kO$--algebras $\kZ \rightarrow \mathit{End}_{\kA^e}(\kA)$, 
where $\kA^e := \kA \otimes_\kO  \kA^\circ$. If $\bbX$ is noetherian then 
it  is an isomorphism.  However, we do not know whether the equality 
$\Gamma(U, \kZ) = 
Z\bigl(A(U)\bigr)$ is true  for an \emph{arbitrary} open subset $U \subseteq X$. Nevertheless, Proposition \ref{P:CentralSheaf} implies that $\bbX$ is central if and only if the canonical morphism $\kO \rightarrow \kZ$ is an isomorphism.
\end{remark}

\begin{remark}\label{R:RemarkCentrality} Let $\bbX = (X, \kA)$ be a ncns. 
Since $\kZ$ is a coherent sheaf of commutative $\kO$--algebras, there exists 
a commutative noetherian scheme $\widetilde{X} = \Spec(\kZ)$ over $X$; see 
\cite[Proposition 1.3.1]{EGAII}. Let $\widetilde{X} \stackrel{\phi}\lar X$ be the corresponding structure  morphism and $\widetilde\kA : = \phi^{-1} \kA$. Then the non--commutative scheme $\widetilde{\bbX} = (\widetilde{X}, \widetilde\kA)$ is central. Moreover, the functor
$
\Qcoh(\widetilde\bbX) \stackrel{\phi_*}\lar \Qcoh(\bbX)
$
is an equivalence of categories. Thus, we get the following important conclusion:
any ncns $\bbX$ can be  replaced  by a \emph{Morita equivalent central} ncns $\widetilde{\bbX}$. 
\end{remark}

\section{Indecomposable injective quasi--coherent sheaves on non--commutative noetherian schemes}

\noindent
The goal of  this section is to clarify  the structure of indecomposable injective objects of  the category $\Qcoh(\bbX)$, where  $\bbX$ is  a ncns. 

\subsection{Prime ideals in non--commutative rings}
Let $A$ be any ring. Unless explicitly stated otherwise, by  an ideal in $A$ we always mean a two--sided ideal. 

\smallskip
\noindent
Recall that an  ideal  $P$ in $A$ is \emph{prime} if for any ideals $I, J$ in $A$ such that 
$I J \subseteq P$ holds: $I \subseteq P$ or  $J \subseteq P$.  Equivalently, for any 
$a, b \in A$ such that $a A b \subseteq I$ we have: $a\in I$ or $b \in I$. We refer to \cite[Proposition 10.2]{LamFirstCourse} for other characterizations of prime ideals in non--commutative rings. Note that any maximal ideal is automatically prime.

\smallskip
\noindent
Similarly to the commutative case, $\Max(A)$ (respectively, 
$\Spec(A)$) denotes the set of maximal (respectively, prime)  ideals in $A$.

\begin{lemma}\label{L:PrimeIdealsSmallStatement} Let $P \in \Spec(A)$ and $I$ be an ideal in $A$ such that $I \not\subseteq P$. Then for any $a \in A \setminus P$ there exists $b \in I$ such that $ba \notin P$.
\end{lemma}
\begin{proof}
Since $I \not\subseteq P$ and $AaA \not\subseteq P$, we conclude that $IaA \not\subseteq P$, hence  $Ia \not\subseteq P$. Therefore, there exists $b \in I$ such that $ba \notin P$.
\end{proof}

\begin{lemma}\label{L:EssExtJacobRadical}
Let $P_1, \dots, P_n \in \Spec(A)$ be such that $P_i \not\subseteq P_j$ for all
$1 \le i \ne j \le n$ and $P := P_1 \cap \dots \cap P_n$. Then the canonical  ring homomorphism
$
A/P \stackrel{\jmath}\lar A/P_1 \times \dots \times A/P_n
$
is an essential extension of $A$--modules. 
\end{lemma}

\begin{proof}
It is sufficient to show that for any 
$0 \ne  x  \in  A/P_1 \times \dots \times A/P_n$, there exists $\lambda \in A$ such that $0 \ne \lambda x \in \mathsf{Im}(\jmath)$. Without loss of generality assume that 
$x = (\bar{a}_1, \dots, \bar{a}_n)$ and 
$a_1 \notin P_1$. Since $P_2 \not\subseteq P_1$, Lemma 
\ref{L:PrimeIdealsSmallStatement} implies that there exists $\mu \in P_2$ such that
$\mu a_1 \notin P_1$. Proceeding inductively, we construct $\lambda \in P_2 \cap \dots \cap P_n$ such that $\lambda a_1 \notin P_1$. Then we get: 
$
\lambda x = (\overline{\lambda a_1}, 0, \dots, 0) \in \mathsf{Im}(\jmath),
$
implying  the statement. 
\end{proof}

\begin{proposition}\label{P:PrimeIdealsBascics} Let $R$ be a commutative noetherian ring and $A$ be a finite $R$--algebra. 
Then the  following statements are true.
\begin{itemize}
\item For any $P \in \Spec(A)$ we have: $P\cap R \in \Spec(R)$.
\item The map $\Spec(A) \stackrel{\varrho}\lar \Spec(R), P \mapsto P\cap R,$ is surjective and has finite fibers. 
\item Let $P \in \Spec(A)$ and $\idp = \varrho(P)$. Then we have: $P_{\idp} \in 
\Max(A_{\idp})$. 
\item Le $\idp, \idq \in \Spec(R)$ and $P \in \Spec(A)$ be such that $\idp \subseteq \idq$ and $\varrho(P) = \idp$. Then there exists $Q \in \Spec(A)$ such that $P \subseteq Q$ and $\varrho(Q) = \idq$. 
\item Let $P, Q \in \Spec(A)$ be such that $P \subseteq Q$ and $\varrho(P) = \varrho(Q)$. Then we have: $P = Q$.
\item Let $\idp \in \Spec(R)$,  $A \stackrel{\jmath}\lar A_{\idp}$  be the canonical ring homomorphism, $Q \in \Max(A_{\idp})$ and $\widetilde{Q} := \jmath^{-1}(Q)$. Then we have: 
$\widetilde{Q} \in \Spec(A)$ and  $\widetilde{Q}_{\idp} = Q$. 
If  $\varrho^{-1}(\idp) = 
\left\{P_1, \dots, P_n\right\}$, where $P_i \ne P_j$ for all $1 \le i \ne j \le n$, then we have: 
$
\Max(A_{\idp}) = \bigl\{(P_1)_{\idp}, \dots,  (P_n)_{\idp}\bigr\}$
 and $\bigl(P_i\bigr)_{\idp} \ne \bigl(P_j\bigr)_{\idp}$ for all $1 \le i \ne j \le n$.
\end{itemize}
\end{proposition}

\smallskip
\noindent
Proofs of all these  results are analogous to the commutative case; see  \cite[Section V.6]{Gabriel}.

\begin{lemma}\label{L:JacobsonRadical}
Let $(R, \idm)$ be a local commutative noetherian ring, $A$ be a finite $R$--algebra, 
$J$ be its Jacobson radical and $\Max(A) = \left\{P_1, \dots, P_n\right\}$. Then we have: $J = P_1 \cap \dots \cap P_n$. 
\end{lemma}

\begin{proof} Recall that $J = \bigcap_{U} \ann_{A}(U)$, where the intersection is taken over the annihilators of all simple left $A$--modules; see 
\cite[Proposition 5.13]{CurtisReiner}. Note that any such $\ann_{A}(U)$ is a prime ideal; see \cite[Proposition 3.15]{GoodearlWarfield}. On the other hand, 
$\idm \subseteq J$; see \cite[Proposition 5.22]{CurtisReiner}; hence 
$\idm \subseteq \ann_{A}(U)$.  By Proposition \ref{P:PrimeIdealsBascics}, $\ann_{A}(U)$
is a maximal ideal in $A$. Conversely, for any  $P \in \Max(A)$ there exists
a simple left $A$--module $U$ such that $P = \ann_{A}(U)$; see \cite[Proposition 3.15]{GoodearlWarfield}. This implies the statement.
\end{proof}

\begin{proposition}\label{P:EndomInjHull}
In the notation of Lemma \ref{L:JacobsonRadical}, let $E_i = E_A(A/P_i)$ be the injective envelope
of $A/P_i$ for all $1 \le i \le n$. Then we have:
$
\End_A(E_1 \oplus \dots \oplus E_n) \cong \widehat{A}^\circ,
$
where $\widehat{A}$ is the $\idm$--adic completion of the algebra $A$. 
\end{proposition}
\begin{proof} Let $E$ be the injective envelope of the left $A$--module $T:= A/J$. Lemma \ref{L:EssExtJacobRadical} implies that $E \cong E_1 \oplus \dots \oplus E_n$. 
Let $\widehat{J}$ be the Jacobson radical of $\widehat{A}$. Then we have:
$\widehat{J} = J \widehat{A}$ and $A/J \cong \widehat{A}/\widehat{J}$. 
The Matlis Duality functor  $\DD$ (see \cite[Corollary 4.3]{Matlis})
establishes an anti--equivalence between the categories of noetherian right $\widehat{A}$--modules and artinian left $\widehat{A}$--modules. Since $T$ 
is semi--simple and of finite length, we have: $\DD(T_{\widehat{A}}) \cong {}_{\widehat{A}}T$. Moreover, 
$\DD$ maps the projective cover of $T$ (which is just $\widehat{A}_{\widehat{A}}$) to 
the injective envelope of  $T$. However, the injective envelope of $T$, viewed as
a left $\widehat{A}$--module,  can be identified with $E$ and 
$\End_{A}(E) \cong \End_{\widehat{A}}(E)$; see e.g. \cite[Theorem 18.6]{Matsumura}
(the proof of \cite{Matsumura} can be literally generalized to the non--commutative setting). Since 
$\End_{\widehat{A}}\bigl(\widehat{A}_{\widehat{A}}\bigr) \cong\widehat{A}$, we conclude that $\End_{A}(E) \cong \widehat{A}^\circ$.
\end{proof}

\subsection{Prime ideals and indecomposable injective modules} Recall the following standard results about indecomposable injective modules. 

\begin{lemma}\label{L:EndomIndInj}
Let $A$ be any ring, $I$ be an injective $A$--module and $H = \End_{A}(I)$. Then the following statements are true.
\begin{itemize}
\item $I$ is indecomposable if and only if $H$ is local. Moreover, in this case 
$f \in H$ is a unit if and only if $\ker(f) = 0$.
\item Assume additionally that $A$ is left noetherian. If $I$ is indecomposable then any $f \in H$ is either a unit or locally nilpotent (i.e.~for any $x \in I$ there exists $n \in \NN$ such that $f^n(x) = 0$).
\end{itemize}
\end{lemma}

\smallskip
\noindent
\emph{Comment to the proof}. For the first statement, see 
\cite[Proposition 2.6]{Matlis}. For the second result, see \cite[Lemme 2, page 428]{Gabriel}. \qed

\smallskip
\noindent
From now on in this subsection, we assume that $R$ is a commutative noetherian ring and $A$ is a finite $R$--algebra. We denote by $\Sp(A)$ the set of  the isomorphism classes of indecomposable injective $A$--modules. 

\begin{proposition}
For any $P \in \Spec(A)$ there exist  uniquely determined $I_P \in \Sp(A)$ and $m_P \in \NN$ such 
that $E_A(A/P) \cong I_P^{\oplus m_P}$. Moreover, the assignment
$$
\Spec(A) \stackrel{\varepsilon}\lar \Sp(A), \;  P \mapsto I_P
$$
is a bijection. 
\end{proposition}

\noindent
\emph{Comment to the proof}. This result is proven in \cite[Section V.4]{Gabriel}. In fact, any indecomposable injective $A$--module $I$ has a uniquely determined  associated prime ideal $P$; see also \cite[Section 3F]{Lam} for further details. \qed

\smallskip
\noindent
Composing the inverse of  $\varepsilon$ with the  map $\varrho$ from Proposition \ref{P:PrimeIdealsBascics}, 
we get a map
$$
\Sp(A) \stackrel{\alpha}\lar \Spec(R). 
$$
It turns out, that  $\alpha$ has a clear conceptual meaning: it assigns to an indecomposable injective $A$--module its uniquely determined associated prime ideal in $R$.

\begin{proposition}\label{P:Associator}
Let $I \in \Sp(A)$ and $\idp = \alpha(I)$. For 
 any $r \in R$, let $\lambda_r^I \in \End_A(I)$ be the (left) multiplication map with $r$. Then the following statements are true. 
 \begin{enumerate}
\item If $r \in  \idp$ then  $\lambda^I_r$ is locally nilpotent, i.e. for any $x \in I$ there exists $n \in \NN$ such that $r^n x = 0$. 
\item If $r \in R\setminus \idp$ then  $\lambda_r^I$ is invertible. 
\item $\idp$ is the unique associated prime ideal of $I$ viewed as an $R$--module. 
\item We have: $\mathsf{Supp}(I) = \overline{\left\{\idp\right\}} \subset \Spec(R)$.
 \end{enumerate}
\end{proposition}
\begin{proof}  Let $P \in \Spec(A)$ be the associated prime ideal of $I$ and
$E$ be the injective envelope of $A/P$. Then there exists $m \in \NN$ such that 
$E \cong I^{\oplus m}$. 
For any $r \in R$, we have a commutative diagram of $A$--modules
$$
\xymatrix{
A/P \ar@{^{(}->}[rr] \ar[d]_-{\lambda_r^{A/P}} & & E \ar[d]^-{\lambda_r^E} \\
A/P \ar@{^{(}->}[rr]  & & E
}
$$
Note that $\lambda_r^E = \mathsf{diag}(\lambda_r^I, \dots, \lambda_r^I)$.

\smallskip
\noindent
(1) If $r\in \idp$ then $\lambda_r^{A/P} = 0$. Hence, $\ker(\lambda_r^E) \ne 0$ and
$\ker(\lambda_r^I) \ne 0$, too. According to Lemma \ref{L:EndomIndInj}, the endomorphism 
$\lambda_r^I$ is locally nilpotent.

\smallskip
\noindent
(2) Let $r\in R\setminus \idp$. Since $P$ is prime, the map $\lambda_r^{A/P}$ is injective. Since the extension $A/P \subset E$ is essential, we have: 
$\ker(\lambda_r^E) =  0$. Hence, $\ker(\lambda_r^I) =  0$ and Lemma \ref{L:EndomIndInj} implies that $\lambda_r^I$ is an isomorphism.

\smallskip
\noindent
(3) We have a non--zero map of $R$--modules $R/\idp \stackrel{\beta}\lar I$, obtained as the composition
$$
R/\idp \xhar A/P \xhar E \dhrightarrow I,
$$
where the last map is an appropriate projection of $E$ onto one of its indecomposable direct summands. It follows from part (2) that $\beta$ is automatically injective, hence
$\idp$ is an associated prime ideal of $R$. 

\smallskip
\noindent
Next, assume that $\idq \ne \idp$ is another associated prime ideal of $I$. Then there exists  an inclusion of $R$--modules $R/\idq \xhar I$. Note that
 for any $r \in R$, the following diagram
$$
\xymatrix{
R/\idq \ar@{^{(}->}[rr] \ar[d]_{\lambda_r^{R/\idq}} & & I \ar[d]^-{\lambda_r^I} \\
R/\idq \ar@{^{(}->}[rr]  & & I 
}
$$
is commutative.  If $r \in \idq \setminus \idp$ then $\lambda_r^{R/\idq} = 0$ and $\lambda_r^I$ is invertible (by part (2)). If $r \in \idp \setminus \idq$ then
$\lambda_r^{R/\idq}$ is injective and $\lambda_r^I$ is locally nilpotent (by part (1)). In both cases, we get a contradiction.

\smallskip
\noindent
(4) The inclusion $\overline{\{\idp\}} \subseteq \mathsf{Supp}(I)$ follows from part (3). If $\idq \in \Spec(R)$ is such that $\idp \not\subseteq \idq$ then there exists
$r \in \idp \setminus \idq$. By part (1), for any $x \in I$ there exists $n \in \NN$ such that $r^n x = 0$. This implies that $I_{\idq} = 0$. \end{proof}

\begin{lemma}\label{L:mapsindinject} Let $I, J \in \Sp(A)$ be such that $\Hom_A(I, J) \ne 0$. Then we have:
$\idp \subseteq \idq$, where $\idp = \alpha(I)$ and $\idq = \alpha(J)$.

\smallskip
\noindent
Conversely, let  $\idp, \idq \in \Spec(R)$ be such that $\idp \subseteq \idq$. Then  there exist $I, J \in \Sp(A)$ such that $\Hom_A(I, J) \ne 0$, $\idp = \alpha(I)$ and
$\idq = \alpha(J)$.
\end{lemma}

\begin{proof}
Let $I \stackrel{f}\lar J$ be a non--zero homomorphism of $A$--modules and 
$x \in I$ be such that $y := f(x) \ne 0$. Assume that there exists $r \in \idp \setminus \idq$. Then $\lambda_r^I$ is locally nilpotent, so we can find  $n \in \NN$ such that
$r^n x = 0$. Hence, $r^n y = 0$, too. On the other hand, the $\lambda_r^J$ is invertible. Contradiction. 

\smallskip
\noindent
To prove the second part, take any $P, Q \in \Spec(A)$ such that $P \subseteq Q$,
$P \cap R = \idp$ and $Q \cap R = \idq$ (such $P$ and $Q$ exist by Proposition \ref{P:PrimeIdealsBascics}). Then we have a non--zero homomorphism of $A$--modules
$A/P \stackrel{g}\lar E_A(A/Q)$, defined as the composition
$
A/P \dhrightarrow A/Q \xhar E_A(A/Q),
$
where $E_A(A/Q)$ is the injective hull of $A/Q$. By injectivity of $E_A(A/Q)$, there exists  a non--zero morphism $E_A(A/P) \stackrel{\tilde{g}}\lar E_A(A/Q)$ extending $g$.
Since $E_A(A/P) \cong I_P^{\oplus m_P}$ and $E_A(A/Q) \cong I_Q^{\oplus m_Q}$ for some $m_P, m_Q \in \NN$, we conclude that $\Hom_A(I_P, I_Q) \ne 0$. 
\end{proof}

\begin{corollary}
For any $\idp \in \Spec(R)$ we put:
$
I(\idp) := \bigoplus\limits_{\substack{I \in \Sp(A) \\ \alpha(I) = \idp}} I.
$
Then for any $\idp, \idq \in \Spec(R)$ we have:
$
\Hom_A\bigl(I(\idp), I(\idq)\bigr) \ne 0$ if and only if
$\idp \subseteq \idq.
$
\end{corollary}

\begin{lemma}
Let $P \in \Spec(A)$, $\idp := P \cap R \in \Spec(R)$ and $E$ be the injective hull of the $A_{\idp}$--module $A_{\idp}/P_{\idp}$. Then we have an isomorphism of 
$A$--modules $E \cong E_A(A/P)$ and $\End_{A_{\idp}}(E) \cong \End_A(E)$. 
\end{lemma}
\begin{proof}
The forgetful functor $A_{\idp}-\Mod \stackrel{\Phi_{\idp}}\lar A-\Mod$ admits an exact left adjoint functor $A-\Mod \lar A_{\idp}-\Mod$ given by the localization with respect to  $\idp$. It is easy to see that the corresponding adjunction counit is an isomorphism. This implies that $\Phi_{\idp}$ is fully faithful and maps injective objects to injective objects. Hence, $E$ is an injective $A$--module and $\End_{A_{\idp}}(E) \cong \End_A(E)$. Next, it is not difficult to see that both inclusions
$
A/P \xhar (A/P)_{\idp} \xhar E
$
are essential extensions of $A$--modules. Hence, $E$ can be identified with the injective hull  of $A/P$, implying the result. 
\end{proof}

\begin{corollary}\label{C:IndInjModules} For any $\idp \in \Spec(R)$, let 
$\overline{\Sp}(A_{\idp})$ be the set of the isomorphism classes of 
indecomposable injective \emph{artinian} $A_{\idp}$--modules. 
Then we  have the following description  of indecomposable injective $A$--modules:
$$
\Sp(A) = \bigsqcup\limits_{\idp \in \Spec(R)} \overline{\Sp}(A_{\idp}) =  \bigsqcup\limits_{\idp \in \Spec(R)} \Max(A_{\idp}),
$$
where we view $I \in \overline{\Sp}(A_{\idp})$ as an element of $\Sp(A)$ via the forgetful functor $\Phi_{\idp}$.
\end{corollary}

\smallskip
\noindent
The following two results play the key role in the proof of the Morita theorem for  ncns.

\begin{proposition}\label{P:NoetherianEndom}
For any $\idp \in \Spec(R)$, the ring  $\End_A\bigl(I(\idp)\bigr)$ is a finite 
$\widehat{R}_{\idp}$--module. In particular, $\End_A\bigl(I(\idp)\bigr)$  is  noetherian. 
\end{proposition}

\begin{proof} In the notation of Proposition \ref{P:PrimeIdealsBascics}, let
$\varrho^{-1}(\idp) = 
\bigl\{P_1, \dots, P_n\bigr\}$, where $P_i \ne P_j$ for all $1 \le i \ne j \le n$.
Then $J := \bigl(P_1\bigr)_{\idp}\cap \dots \cap \bigl(P_n\bigr)_{\idp}$ is the Jacobson radical of $A_{\idp}$. Let $E_i := E_A(A/P_i)$ and 
$E := E_{A_{\idp}}\bigl(A_{\idp}/J\bigr)$. By Proposition \ref{P:EndomInjHull} we have: $\End_{A_{\idp}}(E) \cong {\widehat{A}_{\idp}}^\circ$. 
In particular, $\End_{A_{\idp}}(E)$ is a finite $\widehat{R}_{\idp}$--algebra.
On the other hand, we have
an isomorphism of $A$--modules $E \cong E_1 \oplus \dots \oplus E_n$. Since  $\Phi_{\idp}$ is fully faithful, we get a ring  isomorphism
$$
\End_{A_{\idp}}(E) \cong \End_{A}(E_1 \oplus \dots \oplus E_n).
$$
For any $1 \le i \le n$ there exists $m_i \in \NN$ such that $E_i 
\cong I_{P_i}^{\oplus m_i}$. Therefore, $\End_A\bigl(I(\idp)\bigr)$
and ${\widehat{A}_{\idp}}^\circ$ are Morita--equivalent,  what implies  the statement.
\end{proof}

\begin{lemma}\label{L:KeyLemma}
Let $I, J \in \Sp(A)$ be such that $\Hom_A(I, J) \ne 0$. Assume that $\alpha(I) \ne \alpha(J)$. Then $\Hom_A(I, J)$ is not noetherian viewed as a left $\End_A(J)$--module. 
\end{lemma}
\begin{proof}
Let $\idp = \alpha(I)$ and $\idq = \alpha(J)$. By Lemma \ref{L:mapsindinject} we have: $\idp \subseteq \idq$. Assume that $\Hom_A(I, J)$ is  noetherian viewed as a left $\End_A(J)$--module. By Proposition \ref{P:NoetherianEndom}, there exists a finite map of rings $\widehat{R}_{\idq} \stackrel{\vartheta}\lar \End_A(J)$. Hence, $\Hom_A(I, J)$ is  noetherian viewed as an $\widehat{R}_{\idq}$--module, too. Note that
for any $r \in \idq$, the corresponding element $\vartheta(r) \in \End_A(J)$ acts 
on $\Hom_A(I, J)$ by the rule $f \mapsto  \lambda_r^J \cdot f = f \cdot \lambda_r^I$. 
Suppose now that there exists $r \in \idq \setminus \idp$. Then  
$\lambda_r^I \in \End_A(I)$ is a unit. Hence, $r \cdot \Hom_A(I, J) =
\Hom_A(I, J)$. On the other hand,  $r \in \idq \widehat{R}_{\idq}$.   By Nakayama's Lemma, we get a contradiction.
\end{proof}

\subsection{Indecomposable injective objects of $\Qcoh(\bbX)$}
In this subsection, let $\bbX = (X, \kA)$ be a ncns.
First note the following standard result.

\begin{lemma}\label{L:LocalToGlobal} Let  $U \stackrel{\imath}\xhar X$ be an open subset. 
 Then the  following statements are true.
\begin{itemize}
\item The direct image functor $\Phi_U = \imath_\ast: \Qcoh(\bbU) \lar 
\Qcoh(\bbX)$ is fully faithful and maps (indecomposable) injective objects into (indecomposable) injective objects.
\item Assume that $U$ is affine and $x \in U$. Let $R = \Gamma(U, \kO)$, $A = \Gamma(U, \kA)$, $\idp \in \Spec(R)$ be the prime ideal corresponding to $x$ and
$A_x = A_{\idp}$. Then the functor $$A_x-\Mod \stackrel{\Phi_x}\lar \Qcoh(\bbX)$$ defined as the composition
$
A_x-\Mod \rightarrow A-\Mod \stackrel{\Phi_U}\lar \Qcoh(\bbX),
$
is fully faithful and maps (indecomposable) injective objects into (indecomposable) injective objects.
\item The functor $\Qcoh(\bbX) \lar A_x-\Mod$, assigning to a quasi--coherent $\kA$--module $\kF$ its stalk at the point $x$, is left adjoint
to $\Phi_x$. In particular, the functor $\Phi_x$ does not depend on the choice of an open affine neighbourhood of $x$.
\end{itemize}
\end{lemma}

\smallskip
\noindent
Results from
\cite[Section VI.2]{Gabriel} on Gabriel's recollement of locally noetherian abelian categories, combined with Corollary \ref{C:IndInjModules}, imply the following statement.

\begin{corollary}\label{P:IndInjGlobal}
Let $\Sp(\bbX)$ be the set of the isomorphism classes of indecomposable injective objects of $\Qcoh(\bbX)$. Then we have: 
$$
\Sp(\bbX) = \bigsqcup\limits_{x \in X} \overline{\Sp}(A_x) = \bigsqcup\limits_{x \in X} \Max(A_x),
$$
where we view $I \in \overline{\Sp}(A_{x})$ as an element of $\Sp(\bbX)$ via the  functor $\Phi_{x}$. In particular, we have a surjective map with finite fibers
\begin{equation}\label{E:Associator}
\Sp(\bbX) \stackrel{\alpha}\lar X,
\end{equation}
assigning to $\kI \in \Sp(\bbX)$ the  point $x \in X$ such that 
$\kI \cong \Phi_x(I)$ for some $I \in \overline{\Sp}(A_x)$.
\end{corollary}

\begin{lemma}\label{L:SuppInjSheaf} Let $\kI \in \Sp(\bbX)$ and $x = \alpha(\kI) \in X$. Then we have: $\mathsf{Supp}(\kI) = \overline{\{x\}}$.
\end{lemma}

\begin{proof} First note the following topological fact: a point $y \in X$ belongs to
$\overline{\{x\}}$ if and only if for any open neighbourhood $y \in U \subseteq X$ we have: $x \in U$. 

\smallskip
\noindent
Let $y \in \overline{\{x\}}$. Consider any open affine neighbourhood $x,y \in V$ and put $R := \Gamma(V, \kO)$, $A := \Gamma(V, \kA)$. Let $\idp, \idq \in \Spec(R)$ be the prime ideals corresponding to $x$ and $y$, respectively. 
Note that  $I:= \Gamma(V, \kI)$ is an indecomposable injective $A$--module, whose associated prime ideal is $\idp$. It follows that $\idp \subseteq \idq$, hence
$I_{\idq} \ne 0$. Therefore, we have:  $y \in \mathsf{Supp}(\kI)$. 

\smallskip
\noindent
Assume now that $y \notin \overline{\{x\}}$. Then there exists an open affine subset
$V \subseteq X$ such that $y \in V$ and $x \notin V$.
Proposition \ref{P:Associator} implies that
 for any open affine neighbourhood $x \in U$ we have:
$\Gamma(V, \kI) = \Gamma(V\cap U, \kI) = 0$, hence $\kI_y = 0$. 
\end{proof}

\begin{proposition}\label{P:HomVanishingNonvanishing}
Let $\kI, \kJ \in \Sp(\bbX)$ be such that $\Hom_{\bbX}(\kI, \kJ) \ne 0$, Then we have: $y \in \overline{\{x\}}$, where $x = \alpha(\kI)$ and $y = \alpha(\kJ)$.

\smallskip
\noindent
Conversely, let  $x, y \in X$ be such that $y \in \overline{\{x\}}$. Then  there exist
$\kI, \kJ \in \Sp(\bbX)$ such that 
$x = \alpha(\kI)$, $y = \alpha(\kJ)$ and $\Hom_{\bbX}(\kI, \kJ) \ne 0$.
\end{proposition}

\begin{proof}
Let $J \in \Sp(A_y)$ be such that $\kJ \cong \Phi_y(J)$. Then $
\Hom_{A_y}\bigl(\kI_y, J) \cong \Hom_{\bbX}(\kI, \kJ)  \ne 0,
$
implying that $\kI_y \ne 0$. The first statement is proven. 

\smallskip
\noindent
To show the second part, take any common open affine neighbourhood $x, y\in V$. Let
$R := \Gamma(V, \kO)$ and  $A := \Gamma(V, \kA)$. Let $\idp, \idq  \in \Spec(R)$  be the prime ideals corresponding to the points $x, y \in V$. Then we have: 
$\idp \subseteq \idq$. According to Lemma \ref{L:mapsindinject}, there exist
$I, J \in \Sp(A)$ such that $\Hom_A(I, J) \ne 0$ and $\alpha(I) = \idp, \alpha(J) = \idq$. Let $\kI := \Phi_V(I)$ and $\kJ = \Phi_V(J)$. Then we have: 
$\kI, \kJ \in \Sp(\bbX)$ and $\alpha(\kI) = x, \alpha(\kJ) = y$.
Moreover, since  $\Phi_V$ is fully faithful, we have: $\Hom_{\bbX}(\kI, \kJ) \ne 0$.
\end{proof}

\begin{corollary}
For any $x \in X$ we put:
$
\kI(x) := \bigoplus\limits_{\substack{\kI \in \Sp(\bbX) \\ \alpha(\kI) = x}} \kI.
$
Then for any $x, y \in X$ we have:
$
\Hom_{\bbX}\bigl(\kI(x), \kI(y)\bigr) \ne 0$ if and only if 
$y \in \overline{\{x\}}$.
\end{corollary}

\section{Proof of the Morita theorem for non--commutative noetherian schemes}

\noindent
Let $\bbX = (X, \kA)$ be a ncns. Since we focus on  the study of the category $\Qcoh(\bbX)$, we additionally assume   that $\bbX$ is
\emph{central}; see Remark \ref{R:RemarkCentrality}.

\subsection{Reconstruction of the central scheme} In this subsection we explain, how
the commutative scheme $(X, \kO)$ can be recovered from the category $\Qcoh(\bbX)$.

\begin{lemma}\label{L:KeyNoetherianLemma}
Let $\Lambda$ be a left noetherian ring, $e \in \Lambda$ an idempotent, $\Gamma = e \Lambda e$ and $F = e \Lambda f$, where $f = 1-e$. Then $F$ is noetherian viewed as a left $\Gamma$--module. 
\end{lemma}

\begin{proof}
Let $\widetilde{\Gamma} = f \Lambda f$ and $\widetilde{F} = f \Lambda e$, then we have the Peirce decomposition
$
\Lambda = 
\left(
\begin{array}{cc}
\widetilde{\Gamma} & \widetilde{F} \\
F & \Gamma
\end{array}
\right).
$
Assume that $F$ is not noetherian. Then there exists  an infinite chain of left $\Gamma$--modules
$
F_1 \subsetneqq F_2 \subsetneqq  \dots \subsetneqq F.
$
For any $n \in \NN$, put:
$
J_n :=
\left(
\begin{array}{cc}
\widetilde{F} F_n & \widetilde{F} \\
F_n & \Gamma
\end{array}
\right).
$
Then $J_n$ is a left ideal in $\Lambda$ and we get an infinite chain
$
J_1 \subsetneqq J_2 \subsetneqq  \dots \subsetneqq \Lambda.
$
Contradiction.
\end{proof}

\begin{proposition}\label{P:PropositionReconstructionI}
For any non--empty finite subset $\Omega \subset \Sp(\bbX)$ we put:
$$
\kI(\Omega) := \bigoplus\limits_{\kI \in \Omega} \kI \quad \mbox{\rm and} \quad
A(\Omega):= \End_{\bbX}\bigl(\kI(\Omega)\bigr).
$$
Let  $\Sp(\bbX) \stackrel{\alpha}\lar X$ be the map assigning to an indecomposable injective object of $\Qcoh(\bbX)$ its uniquely determined associated point of $X$ (see Corollary \ref{P:IndInjGlobal}). 
Then the following statements are true. 
\begin{enumerate}
\item If $\Omega = \alpha^{-1}(x)$ for some $x \in X$, then $A(\Omega)$ is noetherian and connected.
\item Conversely, if $A(\Omega)$ is noetherian and connected then we have:
$\big|\alpha(\Omega)\big| = 1$. 
\item Let $\Omega$ be such that $A(\Omega)$ is noetherian and connected, but for any finite $\Omega \subsetneqq \widetilde\Omega$, the algebra 
$A(\widetilde\Omega)$ does not have this property. Then $\Omega = \alpha^{-1}(x)$ for some $x \in X$. 
\end{enumerate}
\end{proposition}
\begin{proof}
(1) Let $x \in X$ and $\Omega = \alpha^{-1}(x)$. By Proposition \ref{P:NoetherianEndom} and Lemma \ref{L:LocalToGlobal}, the algebra $A(\Omega)$ is a finite $\widehat{O}_x$--module, hence it is noetherian. Moreover, it is Morita--equivalent to the algebra 
$\widehat{A}_x^\circ$, hence $\widehat{O}_x = Z(\widehat{A}_x^\circ) =
Z\bigl(A(\Omega)\bigr)$ (at this place we use that $\bbX$ is central). Since the center of a disconnected algebra can not be local, this imples that  
$A(\Omega)$ is connected. The first statement is proven.

\smallskip
\noindent
(2) Now, let  $\Omega \subset \Sp(\bbX)$ be a finite subset such that $\big|\alpha(\Omega)\big| \ge 2$. Choose any  $x \in \Omega$ such that $x \notin \overline{\{y\}}$ for all $y \in \Omega \setminus\{x\}$. For any $\kI, \kJ \in \Omega$ such that $\alpha(\kI) = x$ and  $\alpha(\kJ) \ne x$ we have: $\Hom_{\bbX}(\kJ, \kI) = 0$; see Proposition \ref{P:HomVanishingNonvanishing}. If the algebra $A(\Omega)$ is connected then there exist $\kI, \kJ \in \Omega$ such that $\alpha(\kI) = x$, $\alpha(\kJ) = y \ne x$ and  $\Hom_{\bbX}(\kI, \kJ) \ne 0$. 
Proposition \ref{P:HomVanishingNonvanishing} implies that 
$y \in  \overline{\{x\}}$.

\smallskip
\noindent
Next, there exists an idempotent $e \in A(\Omega)$ such that $e A(\Omega) e \cong 
\End_{\bbX}(\kJ)$. Let $f = 1 -e$. 
Then $\Hom_{\bbX}(\kI, \kJ)$ is a direct summand of
$e A(\Omega) f$ viewed as a left $\End_{\bbX}(\kJ)$--module. 
Suppose now that $A(\Omega)$ is noetherian. Then Lemma \ref{L:KeyNoetherianLemma} implies that $\Hom_{\bbX}(\kI, \kJ)$ is a noetherian left $\End_{\bbX}(\kJ)$--module. 
Let $V \subseteq X$ be an open affine subset such that $x, y \in V$
$R := \Gamma(V, \kO)$ and  $A := \Gamma(V, \kA)$. Then there exist
$I, J \in \Sp(A)$ such that $\kI \cong \Phi_V(I)$ and $\kJ \cong \Phi_V(\kJ)$. 
Moreover,  $\Phi_V$ identifies the left $\End_A(J)$--module  $\Hom_A(I, J)$ with 
the left $\End_{\bbX}(\kJ)$--module $\Hom_{\bbX}(\kI, \kJ)$. However, since $x \ne y$, the associated prime ideals of $I$ and $J$ in the ring $R$ are different. Lemma \ref{L:KeyLemma} implies that 
$\Hom_A(I, J)$
 is not noetherian as a left $\End_A(J)$--module. Contradiction. 
 
 \smallskip
\noindent
(3) This statement is a consequence of the first two. 
\end{proof}

\medskip
\noindent
Proposition \ref{P:PropositionReconstructionI} implies  that the scheme $X$, viewed as a topological space, can be recovered from the category  $\Qcoh(\bbX)$. 
Our next goal is to  explain the reconstruction of the structure sheaf of $X$.
For any closed subset  $Z \subseteq X$ we put:
\begin{equation}
\Qcoh_{Z}(\bbX) = \left\{\kF \in \mathsf{Ob}\bigl(\Qcoh(\bbX)\bigr) \; \big|\; 
\mathsf{Supp}(\kF) \subseteq Z\right\}.
\end{equation}
It is clear that $\Qcoh_{Z}(\bbX)$ is a Serre subcategory of $\Qcoh(\XX)$. 
Let $U:= X\setminus Z \stackrel{\imath}\xhar X$,
 then the restriction functor
$\Qcoh(\bbX) \stackrel{\imath^\ast}\lar \Qcoh(\bbU)$ induces an equivalence of categories
$
\Qcoh(\bbX)/\Qcoh_Z(\bbX) \lar \Qcoh(\bbU).
$
Since $\imath^\ast$ admits a right adjoint functor $\imath_\ast$, $\Qcoh_{Z}(\bbX)$ is a localizing subcategory of $\Qcoh(\bbX)$. 

\smallskip
\noindent
Recall that the localizing subcategories of an arbitrary  locally noetherian abelian category $\sA$ stand in bijection with the subsets of the set $\Sp(\sA)$ of indecomposable injective objects of $\sA$; see \cite[Section III.4]{Gabriel}. Our next goal is to characterize in these terms the localizing subcategories $\Qcoh_{Z}(\bbX)$ for $Z \subseteq X$ closed.

\begin{proposition}\label{P:LocalizingSubcatViaIndInj}
Let $Z \subseteq X$ be a closed subset. Then we have:
$$
\Qcoh_{Z}(\bbX) = \left\{\kF \in \mathsf{Ob}\bigl(\Qcoh(\bbX)\bigr) \; \big|\; 
\Hom_{\bbX}(\kF, \kI) = 0 \; \mbox{\rm for all} \; \kI \in \Sp(\bbX): \alpha(\kI) \in X \setminus Z\right\}.
$$
Let $Z = \overline{\left\{x_1, \dots, x_n\right\}}$, where $x_i \notin \overline{\left\{x_j\right\}}$ for all $1 \le i \ne j \le n$.
 Then $\Qcoh_{Z}(\bbX)$ is the smallest localizing subcategory of $\Qcoh(\bbX)$ containing $\kE := \kI(x_1) \oplus \dots \oplus \kI(x_n)$.
\end{proposition}
\begin{proof} If  $\kI \in \Sp(\bbX)$  is  such that
$x := \alpha(\kI) \in X \setminus Z$ then $\kI \cong \Phi_x(I)$ for some $I \in 
\Sp(A_x)$ and $\Hom_{\bbX}(\kF, \kI) \cong \Hom_{A_x}(\kF_x, I) = 0$ for any $\kF \in \Qcoh_{Z}(\bbX)$. Conversely, let
$\kI \in \Sp(\bbX)$ be such that $x := \alpha(\kI) \in Z$. Then $\kI \cong \Phi_x(I)$ for some $I \in \Sp(A_x)$. Let $P \in \Spec(A_x)$ be the associated prime ideal and
$\kF := \Phi_x(A_x/P)$. Then $\kF \in \Qcoh_{Z}(\bbX)$ and 
$
\Hom_{\bbX}(\kF, \kI)  \cong \Hom_{A_x}(A_x/P, I) \ne 0.
$
The first description of $\Qcoh_{Z}(\bbX)$  follows now from the correspondence between the localizing subcategories
of $\Qcoh(\bbX)$ and the subsets of $\Sp(\bbX)$; see \cite[page 377]{Gabriel}.
To prove the second statement, note that $\mathsf{Supp}(\kE) = Z$; see Lemma \ref{L:SuppInjSheaf}. Let $\sC(\kE)$ be the smallest localizing subcategory of
$\Qcoh(\bbX)$ containing $\kE$. Then it is a subcategory of $\Qcoh_Z(\bbX)$.
According to \cite[Section III.4]{Gabriel}, $\sC(\kE)$ corresponds to a certain subset $\Sigma$
of $\Sp(\bbX)$ containing $\Sigma_Z:= \left\{\kI \in \Sp(\bbX)\; \big|\; \alpha(\kI) \in X \setminus Z\right\}$. Let $\kJ \in \Sp(\bbX) \setminus \Sigma_Z$, i.e.
$\alpha(\kJ) \in Z$. Then Proposition \ref{P:HomVanishingNonvanishing} implies that
$\Hom_{\bbX}(\kE, \kJ) \ne 0$. This shows  that $\Sigma = \Sigma_Z$, hence 
$\sC(\kE) = \Qcoh_Z(\bbX)$, as asserted.
\end{proof}

\begin{theorem}\label{T:MorphismCentralSchemes} Let  $\Qcoh(\bbX) \stackrel{\Phi}\lar \Qcoh(\bbY)$ be an equivalence of categories, where $\bbX = (X, \kA)$ and $\bbY = (Y, \kB)$  are two central ncns. Then there exists a unique isomorphism of schemes
$Y \stackrel{\Phi_c}\lar X$ such that the following diagram of sets
\begin{equation}\label{E:MapCentralSchemes}
\begin{array}{c}
\xymatrix{
\Sp(\bbX) \ar[rr]^-{\widetilde{\Phi}} \ar[d]_-{\alpha_X} & & \Sp(\bbY) \ar[d]^-{\alpha_Y} \\
X  & &  \ar[ll]_-{\Phi_c} Y
}
\end{array}
\end{equation}
is commutative, where  $\widetilde{\Phi}$ is the bijection induced by $\Phi$.
\end{theorem}
\begin{proof} First note that $\alpha_X$ and $\alpha_Y$ are surjective. Hence, $\Phi_c$ is unique (even as a map of sets), provided it exists. 
According to Proposition \ref{P:PropositionReconstructionI}, points of $X$ stand
in bijection with maximal finite subsets $\Omega \subset \Sp(\bbX)$, for which 
the algebra $A(\Omega)$ is connected and noetherian (of course, a similar statement
is true for $\bbY$, too). This shows that there exists a unique bijection
$Y \stackrel{\Phi_c}\lar X$ making the diagram (\ref{E:MapCentralSchemes}) commutative.
Let $x \in X$ and $y = \Phi_c^{-1}(x)$.  Proposition \ref{P:HomVanishingNonvanishing} implies that $\Phi_c^{-1}\bigl(\overline{\left\{x\right\}}\bigr) = \overline{\left\{y\right\}}$. Hence, the map $\Phi_c$ is continuous.

\smallskip
\noindent
Let $Z \subseteq X$ be any closed subset and $W:= \varphi^{-1}(Z)$. We put: 
$U:= U \setminus Z$ and $V:= Y \setminus W$.  Then we have a commutative diagram of categories and functors
$$
\xymatrix{
\Qcoh_Z(\bbX) \ar@{_{(}->}[d] \ar[rr]^-{\Phi_{\mid}} & & \Qcoh_W(\bbY) \ar@{^{(}->}[d] \\
\Qcoh(\bbX) \ar@{->>}[d] \ar[rr]^-{\Phi} & & \Qcoh(\bbY) \ar@{->>}[d] \\
\Qcoh(\bbU) \ar[rr]^-{\Phi_U} & & \Qcoh(\bbV) 
}
$$
where $\Phi_{\mid}$ and $\Phi_U$ denote the restricted and induced equivalences of the corresponding categories. Let
$Z\bigl(\Qcoh(\bbU)\bigr) \stackrel{\psi_U}\lar  Z\bigl(\Qcoh(\bbV)\bigr)$ be the map of centers induced by $\Phi_U$. The characterization of the subcategory
$\Qcoh_Z(\bbX)$ in the terms of indecomposable injective objects (see
Proposition \ref{P:LocalizingSubcatViaIndInj}) combined with Corollary \ref{C:Centers} imply that the collection of ring isomorphisms $(\psi_U)_{U \subseteq X}$ defines
a sheaf isomorphism $\kO_X \rightarrow (\Phi_c)_\ast\kO_Y$. Hence, $Y \stackrel{\Phi_c}\lar X$ is an isomorphism of schemes, as asserted.
\end{proof}

\smallskip
\noindent
\textbf{Summary} (Reconstruction of the central scheme). Let $\bbX = (X, \kA)$ be a central ncns. 
\begin{itemize}
\item Consider the set $S = S(\bbX)$, whose elements  are maximal finite subsets $\Omega \subset
\Sp(\bbX)$ such that the algebra $A(\Omega)$ is noetherian and connected. 
\item Define the topology on $S$ by the following rules:
\begin{itemize}
\item For any  $\Omega', \Omega'' \in S(\bbX)$ we say that
 $$\Omega'' \in \overline{\{\Omega'\}} \quad \mbox{\rm if and only if}\quad 
\Hom_{\bbX}\bigl(\kI(\Omega'), \kI(\Omega'')\bigr) \ne 0.
$$
\item By definition, any non--trivial closed subset of $S$ has the form
$$
Z(\Omega_1, \dots, \Omega_n) :=  \bigcup\limits_{i = 1}^n \overline{\{\Omega_i\}},
$$
where $\Omega_1, \dots, \Omega_n \in S$ are such that $\Omega_i \notin 
\overline{\{\Omega_j\}}$ for all $1 \le i \ne j \le n$.
\end{itemize}
\item For any open subset  $U = U(\Omega_1, \dots, \Omega_n) := S \setminus Z(\Omega_1, \dots, \Omega_n)$, let $\sC(Z)$ be the smallest localizing subcategory of $\Qcoh(\bbX)$ containing $\kI(\Omega_1) \oplus \dots \oplus \kI(\Omega_n)$ and 
$
\Gamma(U, \kR) :=   Z\bigl(\Qcoh(\bbX)/\sC(Z)\bigr)
$
\item Finally, let $\widetilde{U} \subseteq U$ be a pair of open subsets, $Z := X \setminus U$,  $\widetilde{Z} := X \setminus \widetilde{U}$ and 
$\Gamma(U, \kR) \rightarrow \Gamma(\widetilde{U}, \kR)$ be the  map of centers
induced by  the localization functor $\Qcoh(\bbX)/\sC(Z) \rightarrow \Qcoh(\bbX)/\sC(\widetilde{Z}).$
\end{itemize}
Then  $S$ is a topological space and $\kR$ is a sheaf of commutative rings on 
$S$. Moreover, the  ringed spaces  $(X, \kO)$ and
$(S, \kR)$ are isomorphic. The corresponding  isomorphism of topological spaces $X \rightarrow S(\bbX)$ is given by the rule $x \mapsto \alpha^{-1}(x)$, where $\alpha$ is the map given by (\ref{E:Associator}). In other words, the underlying commutative  scheme $X$ of a central non--commutative noetherian scheme $\bbX$ can be recovered from the category $\Qcoh(\bbX)$. This provides a generalization of the  classical reconstruction result of Gabriel \cite[Section VI.3]{Gabriel} on the non--commutative setting.

\subsection{Proof of Morita theorem}
For a ncns $\bbX = (X, \kA)$ we put: $\bbX^\circ := (X, \kA^\circ)$.  Next,
$\VB(\bbX)$ (respectively, $\VB(\bbX^\circ)$) will denote the category  of coherent sheaves on $\bbX$ (respectively, $\bbX^\circ$) which are locally projective over $\kA$ (respectively, over $\kA^\circ$).

\begin{definition}
Let $\bbX$ be a ncns. Then $\kP \in \VB(\bbX^\circ)$ is a  \emph{local right progenerator} of $\bbX$ if $\add(\kP_x) = \add(A_x)$  for all $x \in X$. 
\end{definition}

\begin{theorem}\label{T:MoritaTheoremNCNS}
Let  $\Qcoh(\bbX) \stackrel{\Phi}\lar \Qcoh(\bbY)$ be an equivalence of categories, where $\bbX = (X, \kA)$ and $\bbY = (Y, \kB)$ are  central ncns. Then there exist
a pair $(\kP, \vartheta)$, where 
\begin{itemize}
\item $\kP \in \VB(\bbX^\circ)$ is a local right progenerator
\item  $\kB \stackrel{\vartheta}\lar 
\varphi^\ast\bigl(\mathit{End}_{\kA}(\kP)\bigr)$ is an isomorphism of $\kO_Y$--algebras
\end{itemize}
such that $\Phi \cong  \Phi_{\kP, \vartheta, \varphi} := 
\vartheta^{\sharp}\cdot \varphi^\ast \cdot \bigl(\kP \otimes_\kA \,-\, \bigr)$, where $\varphi = \Phi_c: Y \lar X$ is the scheme isomorphism induced by the equivalence $\Phi$ (see Theorem \ref{T:MorphismCentralSchemes}) and $\vartheta^{\sharp}$ is the equivalence of categories  induced by $\vartheta$.

\smallskip
\noindent
If $(\kP', \vartheta')$ is another pair representing $\Phi$ (i.e $\Phi \cong  \Phi_{\kP', \vartheta', \varphi}$)
then there exists a unique isomorphism $\kP \stackrel{f}\lar \kP'$ in $\VB(\bbX^\circ)$ such that the diagram
$$
\xymatrix{
& \kB \ar[ld]_-{\vartheta} \ar[rd]^-{\vartheta'} & \\
\varphi^\ast\bigl(\mathit{End}_{\kA}(\kP)\bigr) \ar[rr]^-{\varphi^\ast(Ad_f)} & & \varphi^\ast\bigl(\mathit{End}_{\kA}(\kP')\bigr)
}
$$
is commutative.

\smallskip
\noindent
Conversely, if $Y \stackrel{\varphi}\lar X$ is an isomorphism of schemes,  $\kP \in \VB(\bbX^\circ)$ is a local right progenerator and $\kB \stackrel{\vartheta}\lar 
\varphi^\ast\bigl(\mathit{End}_{\kA}(\kP)\bigr)$ is an isomorphism of $\kO_Y$--algebras
then $\Phi:= \Phi_{\kP, \vartheta, \varphi}$ is an equivalence of categories such that $\Phi_c = \varphi$.
\end{theorem}

\begin{proof} The last part of the theorem  is obvious. Hence, it is sufficient to prove the following 

\smallskip
\noindent
\underline{Statement}. Let  $\Qcoh(\bbX) \stackrel{\Phi}\lar \Qcoh(\widetilde\bbX)$ be a central equivalence of categories, where  $\bbX = (X, \kA)$ and $\widetilde\bbX = (X, \widetilde\kA)$ are  two central ncns with the same underlying commutative scheme $X$. Then $\Phi \cong  \kP \otimes_\kA \,-$, where $\kP \in  \VB(\bbX^\circ)$
is a balanced central $(\widetilde\kA-\kA)$--bimodule  which is a local right progenerator of $\bbX$. 

\smallskip
\noindent
\underline{Claim}. 
For any open subset 
 $U \stackrel{\imath}\xhar X$ put:  $
 \Phi_U:= \imath^\ast \Phi \imath_\ast: \Qcoh(\bbU) \lar
\Qcoh(\widetilde\bbU).$
Then $\Phi_U$ is an equivalence of categories and in the  following diagram of categories and functors
\begin{equation}\label{E:inducedEq}
\begin{array}{c}
\xymatrix{
\Qcoh(\bbX) \ar[rr]^-{\imath^\ast} \ar[d]_-{\Phi} & & \Qcoh(\bbU)
\ar[d]^-{\Phi_U} \\
\Qcoh(\widetilde{\bbX}) \ar[rr]^-{\imath^\ast} & & \Qcoh(\widetilde\bbU)
}
\end{array}
\end{equation}
both compositions of functors are isomorphic. 

\smallskip
\noindent
Indeed, let $Z:= X \setminus U$. Then 
$\Phi$ restricts to  an equivalence of the categories 
$\Qcoh_Z(\bbX) \rightarrow \Qcoh_Z(\widetilde{\bbX})$ (at this place, we use \emph{centrality} of $\Phi$). The universal property of the Serre quotient category implies that there exists an equivalence of categories $\Qcoh(\bbU) \stackrel{\Psi_U}\lar \Qcoh(\widetilde\bbU)$ such that $\Psi_U \imath^\ast \cong  \imath^\ast \Phi$.
Since $\imath^\ast \imath_\ast =  \mathsf{Id}_{\bbU}$, we conclude that $\Phi_U \cong \Psi_U$, hence $\Phi_U$ is an equivalence of categories.  One can check that the natural transformation 
$
\imath^\ast \Phi \stackrel{\zeta_U}\lar \Phi_U \imath^\ast,
$
induced by the adjunction unit $\mathsf{Id}_{\bbX} \rightarrow \imath_\ast \imath^\ast$, is an isomorphism. This proves the claim.

\smallskip
\noindent
Let $V \stackrel{\varepsilon}\xhar U$ be an open subset and 
$\jmath:= \imath \varepsilon$. Since $\imath_\ast \varepsilon_\ast = \jmath_\ast$ and
$\varepsilon^\ast \imath^\ast = \jmath^\ast$, we conclude:
\begin{equation}\label{E:EquivRestricted}
\Phi_V = \varepsilon^\ast \Phi_U \varepsilon_\ast.
\end{equation}

\smallskip
\noindent
Assume now that $U$ is affine. Then there exists a central $(\widetilde\kA_U-\kA_U)$--Morita bimodule $\kP^U \in \VB(\bbU^\circ)$ and an isomorphism of functors
$
\Phi_U \stackrel{\xi_U}\lar \kP^U \otimes_{\kA_U} \,-\,.
$
Then for  any $\kG \in \Qcoh(\bbV)$, we get natural isomorphisms
\begin{equation}\label{E:SomeNatTransf}
\varepsilon^\ast \Phi_U \varepsilon_\ast(\kG) \lar 
\varepsilon^\ast\bigl(\kP^U \otimes_{\kA_U} \varepsilon_\ast(\kG)\bigr) \lar
\kP^U\Big|_{V} \otimes_{\kA_V} \kG,
\end{equation}
where we use that $\varepsilon^\ast \varepsilon_\ast = \mathsf{Id}_{\bbV}$.

\smallskip
\noindent
Let $\kP^U\big|_{V} \otimes_{\kA_V} \,-\, \xrightarrow{\widetilde{\sigma}^U_V}
\kP^V \otimes_{\kA_V} \,-\,$ be the unique isomorphism  of functors making the following diagram of functors and natural transformations
\begin{equation}\label{E:MoritaCompat}
\begin{array}{c}
\xymatrix{
\varepsilon^\ast \Phi_U \varepsilon_\ast \ar[rr]^-{\xi_U\big|_V} \ar[d]_{=} & & \kP^U\big|_{V} \otimes_{\kA_V} \,-\,  \ar[d]^-{\widetilde{\sigma}^U_V} \\
\Phi_V \ar[rr]^-{\xi_V} & & \kP^V \otimes_{\kA_V} \,-\,
}
\end{array}
\end{equation}
commutative, where $\xi_U\big|_V$ is the isomorphism of functors defined by (\ref{E:SomeNatTransf}). According to Theorem \ref{T:MoritaClassical} (classical Morita theorem for rings), there exists a uniquely determined  isomorphism of $(\widetilde\kA_V-\kA_V)$--bimodules 
$
\kP^U\big|_{V} \stackrel{\sigma^U_V}\lar \kP^V,
$
which induces the natural transformation $\widetilde{\sigma}^U_V$. It follows from 
(\ref{E:EquivRestricted}) and (\ref{E:MoritaCompat}) that for any triple
$W \subseteq V \subseteq U$ of open affine subsets of $X$, the following diagram
$$
\xymatrix{
\bigl(\kP^U\big|_V\bigr)\Big|_{W} \ar[rr]^-{(\sigma^U_V)\big|_{W}} \ar[d]_= & & 
\kP^V\big|_{W} \ar[d]^-{\sigma^V_W} \\
\kP^U\big|_W \ar[rr]^-{\sigma^U_W} & & \kP^W
}
$$
is commutative. Hence, there exists an $(\widetilde\kA-\kA)$--bimodule $\kP$ 
and a family of isomorphisms of $(\widetilde\kA_U-\kA_U)$--bimodules $\kP^U \stackrel{\varphi_U}\lar \kP\big|_U$ (for any $U \subseteq X$ open and affine) such that
$$
\xymatrix{
\kP^U\big|_V  \ar[rr]^-{(\varphi_U)\big|_{V}} \ar[d]_-{\sigma^U_V} & & 
  \bigl(\kP\big|_{U}\bigr)\big|_{V} \ar[d]^-{=}\\
\kP^V  \ar[rr]^-{\varphi_V} & & \kP\big|_{V}
}
$$
is commutative for any pair of open affine subsets $V \subseteq U$ of $X$. It is clear  that
$\kP$ is a central balanced $(\widetilde\kA-\kA)$--bimodule such that
$\add(\kP_x) = \add(A_x)$ for any $x \in X$. Moreover, the datum 
$\bigl(\kP, (\varphi_U)_{U \subseteq X}\bigr)$ is unique up to a unique isomorphism. 

\smallskip
\noindent
Finally, for any $\kF \in \Qcoh(\bbX)$ and $U \subseteq X$ open and affine, we have an isomorphism $\Phi(\kF)\big|_{U} \lar \bigl(\kP \otimes_\kA \kF\bigr)\big|_U$ defined as the composition
$$
\Phi(\kF)\big|_{U} \stackrel{\zeta^\kF_U}\lar \Phi_U\bigl(\kF\big|_U\bigr)
\xrightarrow{\xi_U^{\kF\big|_U}} \kP^U \otimes_{\kA_U} \kF\big|_U 
\xrightarrow{\varphi_U \otimes \mathsf{id}} \kP\big|_U \otimes_{\kA_U} \kF\big|_U 
\stackrel{\mathsf{can}}\lar \bigl(\kP \otimes_\kA \kF\bigr)\big|_U.
$$
It follows that these isomorphisms are compatible with restrictions on open affine subsets and define a global isomorphism of left $\widetilde\kA$--modules
$\Phi(\kF) \stackrel{\vartheta^\kF}\lar \kP \otimes_\kA \kF$, which is natural in $\kF$. Hence, we have constructed an isomorphism of functors  $\Phi \cong  \kP \otimes_\kA \,-$ we were looking for. The uniqueness of $\kP$ follows from the corresponding  result in the  affine case. 
\end{proof}

\begin{remark}
In the case when $\bbX = (X, \kA)$ and $\bbY = (Y, \kB)$ are ncns with $X$ and $Y$ 
being  locally of finite type over a field $\kk$, some related results about  equivalences between the categories $\Coh(\bbX)$ and $\Coh(\bbY)$ can also be found in \cite[Section 6]{ArtinZhang}.
\end{remark}

\section{C\u{a}ld\u{a}raru's conjecture on Azumaya algebras on noetherian schemes}

\noindent
Let $X$ be a  noetherian scheme and $\kA$ be a sheaf of $\kO$--algebras, which is a locally free coherent sheaf of finite rank on $X$. Then we have a canonical morphism of $\kO$--algebras 
$
\kA \otimes_\kO \kA^\circ \stackrel{\mu}\lar \mathit{End}_{\kO}(\kA),
$
given on the level of local sections by the rule $a \otimes b \mapsto (c \mapsto a cb)$. Recall that $\kA$ is an \emph{Azumaya algebra} on $X$ if $\mu$ is an isomorphism. This  is equivalent to the condition that $\kA\big|_x := A_x \otimes_{O_x} \bigl(O_x/\idm_x\bigr)$  is a central simple $O_x/\idm_x$--algebra for any point $x \in X$; see \cite[Proposition IV.2.1]{Milne}. It follows that an Azumaya algebra on $X$ is automatically central. Moreover, for any pair of Azumaya algebras $\kA_1$ and  
$\kA_2$ on $X$, their tensor product $\kA_1 \otimes_\kO \kA_2$ is again an Azumaya
algebra.

\smallskip
\noindent
Let $\kA$ be an Azumaya algebra on $X$, $\bbX = (X, \kA)$ and $\kP \in \VB(\bbX^\circ)$ be a local right progenerator of $\bbX$. Then $\widetilde\kA := \mathit{End}_\kA(\kP)$ is again an Azumaya algebra on $X$ (since $\widetilde{\kA}\big|_x$ is again a central simple $O_x/\idm_x$--algebra algebra for any $x \in X$). If 
$\widetilde{\bbX} = (X, \widetilde\kA)$ then 
$$
\Qcoh(\bbX) \xrightarrow{\kP \otimes_\kA \,-\,} \Qcoh(\widetilde\bbX)
$$
is a central equivalence of categories. 

\smallskip
\noindent
Let  $\kA$ and $\kB$ be two Azumaya algebras on $X$. We put:  $\bbX = (X, \kA)$ and $
\widetilde\bbX = (X, \kB)$. 
\begin{itemize}
\item $\kA$ and $\kB$ are centrally Morita equivalent (denoted  $\kA \approx \kB$) if there exists a central  equivalence of categories 
$\Qcoh(\bbX) \rightarrow \Qcoh(\widetilde\bbX)$
\item $\kA$ and $\kB$ are  \emph{equivalent} (denoted $\kA \sim \kB$) provided there exist $\kF, \kG \in \VB(X)$ such that 
$\mathit{End}_{\kO}(\kF) \otimes_\kO \kA$ and $\mathit{End}_{\kO}(\kG)
\otimes_\kO \kB$ are isomorphic as  $\kO$--algebras.
\end{itemize}
 The second relation  is indeed an equivalence relation. The set $\mathsf{Br}(X)$ of equivalence classes of Azumaya algebras on $X$, endowed with the operation
$$
[\kA_1] + [\kA_2] := \bigl[\kA_1 \otimes_\kO \kA_2\bigr]
$$
is a commutative group   (called \emph{Brauer group} of the scheme $X$). Recall that 
 the class $[\kO]$  is the neutral element of $\mathsf{Br}(X)$, whereas 
 $-[\kA] = [\kA^\circ]$;  see \cite[Section IV.2]{Milne}. 

\begin{lemma}\label{L:LemmaAzumayaEq}
Let $\kA, \kB$ be Azumaya algebras on $X$ such that $\kA \approx \kB$. Then for any Azumaya algebra $\kC$ on $X$ we have:
$\kA \otimes_\kO \kC \approx \kB \otimes_\kO \kC$.
\end{lemma}

\begin{proof}
Since $\kA \approx \kB$, there exists a local right progenerator $\kP$ for $\kA$ and an isomorphism of $\kO$--algebras $\kB \stackrel{\vartheta}\lar \mathit{End}_\kA(\kP)$. Then we get an induced isomorphism of $\kO$--algebras $\kB \otimes_\kO \kC 
\stackrel{\tilde\vartheta}\lar \mathit{End}_{\kA\otimes_\kO \kC}(\kP\otimes_\kO \kC)$ given as the composition
$$
\kB \otimes_\kO \kC \xrightarrow{\vartheta \otimes \mathsf{id}} \mathit{End}_\kA(\kP)
\otimes_\kO \kC \lar \mathit{End}_{\kA\otimes_\kO \kC}(\kP\otimes_\kO \kC).
$$
Note  that $\kP\otimes_\kO \kC$ is a local right progenerator for $\kA\otimes_\kO \kC$. Moreover, $\kO$ is the center of $\kA \otimes_\kO \kC$ and 
 $\kB \otimes_\kO \kC$ (since they both are Azumaya algebras), implying  the statement.
\end{proof}

\smallskip
\noindent
The proof of the following result is basically a replica of 
\cite[Theorem 1.3.15]{CaldararuThesis}.

\begin{proposition}\label{P:Azumaya}
Let $\kA$ and $\kB$ be two Azumaya algebras on $X$. Then we have: 
$$
\kA \sim \kB \quad \mbox{\rm if and only if} \quad \kA \approx \kB.
$$
\end{proposition}

\begin{proof} Let $\bbX = (X, \kA)$ and $\bbY = (X, \kB)$
If $\kA \sim \kB$ then there exist $\kF, \kG \in \VB(X)$ and an isomorphism of $\kO$--algebras
$
\mathit{End}_{\kO}(\kF) \otimes_\kO \kA \stackrel{\vartheta}\lar \mathit{End}_{\kO}(\kG)
\otimes_\kO \kB.
$
Let $\kP := \kF \otimes_\kO \kA$ and $\kQ := \kG \otimes_\kO \kB$. Then $\kP$ is a local right progenerator for $\kA$ and $\kQ$ is a local right progenerator for $\kB$.
Let $\widetilde\kA:= \mathit{End}_{\kA}(\kP)$ and $\widetilde\kB:= \mathit{End}_{\kB}(\kQ)$. Then we have  isomorphisms of $\kO$--algebras 
$\widetilde\kA \cong \mathit{End}_{\kO}(\kF) \otimes_\kO \kA$ and 
$\widetilde\kB \cong \mathit{End}_{\kO}(\kG) \otimes_\kO \kB$ as well as central equivalences of categories
$$
\Qcoh(\bbX) \xrightarrow{\kP \otimes_\kA \,-\,} \Qcoh(\widetilde\bbX)
\xleftarrow{\vartheta^\sharp} \Qcoh(\widetilde\bbY) 
\xleftarrow{\kQ \otimes_\kB \,-\,}
\Qcoh(\bbY),
$$
where $\widetilde\bbX = (X, \widetilde\kA)$ and 
$\widetilde\bbY = (X, \widetilde\kB)$. Hence, $\kA \approx \kB$.

\smallskip
\noindent
Conversely, assume that $\kA \approx \kB$. Then we have:
$$
\kO \approx \mathit{End}_{\kO}(\kA) \cong \kA \otimes_\kO \kA^\circ \approx \kB \otimes_\kO \kA^\circ,
$$
where the last central equivalence exists by Lemma \ref{L:LemmaAzumayaEq}. Hence, there exists $\kF \in \VB(X)$ and an isomorphism of $\kO$--algebras
$\mathit{End}_{\kO}(\kF) \stackrel{\vartheta}\lar \kB \otimes_\kO  \kA^\circ$. Then we get the following induced isomorphism of $\kO$--algebras:
$$
\mathit{End}_{\kO}(\kF) \otimes_\kO \kA \xrightarrow{\vartheta \otimes \mathsf{id}}
\kB \otimes_\kO \kA^\circ \otimes_\kO \kA \stackrel{\mathsf{can}}\lar 
\kA \otimes_\kO \kA^\circ \otimes_\kO \kB \xrightarrow{\mu \otimes \mathsf{id}} \mathit{End}_{\kO}(\kA) \otimes_\kO \kB.
$$
Hence, $\kA \sim \kB$, as asserted. 
\end{proof}

\begin{theorem}\label{T:CaldararuConjecture}
Let $X$ and $Y$ be two separated noetherian schemes, $\kA$ be an Azumaya algebra on $X$, $\kB$ be an Azumaya algebra on $Y$, $\bbX = (X, \kA)$ and $\bbY = (Y, \kB)$.
Then the categories $\Qcoh(\bbX)$ and $\Qcoh(\bbY)$ are equivalent if and only if there exists an isomorphism of schemes $Y \stackrel{f}\lar X$ such that
$f^\ast\bigl([\kA]\bigr) = [\kB] \in \mathsf{Br}(Y)$. 
\end{theorem}

\begin{proof}
If $Y \stackrel{f}\lar X$ is such that
$f^\ast\bigl([\kA]\bigr) = [\kB] \in \mathsf{Br}(X)$ then equivalence of $\Qcoh(\bbX)$ and $\Qcoh(\bbY)$ is a consequence of Proposition \ref{P:Azumaya}.

\smallskip
\noindent
Conversely, let $\Qcoh(\bbX) \stackrel{\Phi}\lar \Qcoh(\bbY)$ be an equivalence of categories. Since both ncns $\bbX$ and $\bbY$ are central, Theorem \ref{T:MorphismCentralSchemes} yields  an induced isomorphism of schemes 
$Y \stackrel{f}\lar X$, where $f = \Phi_c$. It follows that $\kB \approx f^\ast(\kA)$. By Proposition \ref{P:Azumaya} we get: $\kB \sim f^\ast(\kA)$.
\end{proof}

\begin{remark}
Theorem \ref{T:CaldararuConjecture} was conjectured by C\u{a}ld\u{a}raru in 
\cite[Conjecture 1.3.17]{CaldararuCrelle}. In the case of regular projective varieties over a field, it was proved by Canonaco and Stellari \cite[Corollary 5.3]{CanonacoStellari}. In the full generality, C\u{a}ld\u{a}raru's conjecture was proven by  Antieau \cite[Theorem 1.1]{Antieau}, based on a previous work of Perego \cite{Perego} and the theory of derived Azumaya algebras of To\"en \cite{Toen}. In our opinion, the given proof of C\u{a}ld\u{a}raru's conjecture (in which Theorem \ref{T:MorphismCentralSchemes} plays a  key role) is significantly  simpler.
\end{remark}

\section{Local modification theorem}

\noindent
In this section, let $R$ be a connected reduced excellent commutative ring of Krull dimension one and $K = \Quot(R)$ be its total ring of fractions (which is isomorphic to a finite product of fields). For any $\idm \in \Max(R)$, we have a multiplicatively closed set $S_{(\idm)} := R \setminus \{\idm\} \subset R$. Then we have: $K_{\idm} := S_{(\idm)}^{-1} K \cong \Quot(R_{\idm})$. Next, we have a commutative diagram of rings and canonical ring homomorphisms
$$
\xymatrix{
R \ar[r] \ar@{_{(}->}[d] & R_{\idm} \ar[r] \ar@{_{(}->}[d] & \widehat{R}_{\idm} \ar@{^{(}->}[d] \\
K \ar[r] & K_{\idm} \ar[r] & \widehat{K}_{\idm}
}
$$
where $\widehat{K}_{\idm} = \Quot(\widehat{R}_{\idm})$. Since $R$ is excellent, the completion $\widehat{R}_{\idm}$ is a reduced ring (see e.g. \cite{Dieudonne}) and  $\Quot(\widehat{R}_{\idm})$ is isomorphic to a finite product of fields. Note that the canonical ring homomorphism  $K \otimes_R \widehat{R}_{\idm} 
\rightarrow \widehat{K}_{\idm}$ is an isomorphism. 

\smallskip
\noindent
Let $U$ be a finitely generated $K$--module and $L \subset U$ a finitely generated $R$--submodule such that $K \cdot L = U$. Then $L$ is automatically a torsion free $R$--module and the  canonical map $K \otimes_R L \rightarrow U$ is an isomorphism. We shall also say that $L$ is an \emph{$R$--lattice} in (its \emph{rational envelope})
$U$. 

\smallskip
\noindent
For any $\idm \in \Max(R)$ we put: 
$U_{\idm} = K_{\idm} \otimes_K U$ and $\widehat{U}_{\idm} := \widehat{K}_{\idm} \otimes_K U$  as well as
$L_{\idm} := R_{\idm} \otimes_R L$ and 
$\widehat{L}_{\idm} := \widehat{R}_{\idm} \otimes_R L$. Using the canonical
ring homomorphisms $K \otimes_R R_{\idm} 
\rightarrow K_{\idm}$ and $K \otimes_R \widehat{R}_{\idm} 
\rightarrow \widehat{K}_{\idm}$,
we can view $L_{\idm}$ as an $R_{\idm}$--lattice in $U_{\idm}$ and 
 $\widehat{L}_{\idm}$ as an $\widehat{R}_{\idm}$--lattice in 
$\widehat{U}_{\idm}$. Next, we  have the following commutative diagram:
\begin{equation}
\begin{array}{c}
\xymatrix{
L \ar[r] \ar@{_{(}->}[d] &
L_{\idm}  \ar@{^{(}->}[r] \ar@{_{(}->}[d]^-{\varepsilon_{\idm}}  &  
\widehat{L}_{\idm} \ar@{^{(}->}[d]^-{\hat{\varepsilon}_{\idm}} \\
U \ar[r]^-{\vartheta_{\idm}} \ar@/_5.5ex/[rr]_-{\hat{\vartheta}_{\idm}} & 
U_{\idm}  \ar@{^{(}->}[r] &  \widehat{U}_{\idm} 
}
\end{array}
\end{equation}
in which all maps  are the canonical ones.

\begin{lemma}\label{L:LatticeRationalLocal} In the above notation we have:
\begin{equation}\label{E:LatticeRationalLocal}
L = \left\{
x \in U \, \big| \, \hat{\vartheta}_{\idm}(x) \in \mathsf{Im}(\hat{\varepsilon}_{\idm})
\; \mbox{\rm for all} \; \idm \in \Max(R)
\right\}.
\end{equation}
\end{lemma}
\begin{proof} For any $\idm \in \Max(R)$ we have: $L_{\idm} = \widehat{L}_{\idm} \cap U_{\idm}$, where the intersection is taken inside  $\widehat{U}_{\idm}$; see for instance \cite[Theorem 5.2]{ReinerMO}. Hence, it is sufficient to show that
$$
L = \widetilde{L} := \left\{
x \in U \, \big| \, \vartheta_{\idm}(x) \in \mathsf{Im}(\varepsilon_{\idm})
\; \mbox{\rm for all} \; \idm \in \Max(R)
\right\}.
$$
It is clear that $L \subseteq \widetilde{L}$, hence we only need to prove the opposite inclusion. Let $x \in \widetilde{L}$ and
$
I:= \left\{
a \in R \, \big| \, a x \in L
\right\}.
$
By definition of $\widetilde{L}$, for any $\idm \in \Max(R)$ there exists
$t \in S_{(\idm)}$ such that $t x \in L$. Since $t \in I \setminus \idm$, we conclude that $I \not\subset \idm$ for any $\idm \in \Max(R)$. As a consequence, 
$I = A$ and $x \in L$, as asserted.
\end{proof}

\begin{theorem}[Local modification theorem]\label{T:LocalModification} Let $U$ be a finitely generated $K$--module, $L \subset U$ an $R$--lattice and  $\Omega \subset \Max(R)$  a finite subset such that for any $\idm \in \Omega$ we are given an 
$\widehat{R}_{\idm}$--lattice $N(\idm) \subset \widehat{U}_{\idm}$. Then there exists a unique lattice $N \subset U$ (\emph{local modification} of $L$) such that for
any $\idm \in \Max(R)$ we have: 
\begin{equation}\label{E:prescribedcompletions}
\widehat{N}_{\idm} =
\left\{
\begin{array}{cc}
\widehat{L}_{\idm} & \mbox{\rm if} \;  \idm \notin \Omega \\
N(\idm) & \mbox{\rm if} \;  \idm \in \Omega,
\end{array}
\right.
\end{equation}
where $\widehat{N}_{\idm}$ is viewed as a subset of $\widehat{U}_{\idm}$.
\end{theorem}

\begin{proof} In order to prove the existence of $N$, we first consider the following special

\noindent
\underline{Case 1}. Suppose   that $N(\idm) \subseteq \widehat{L}_{\idm}$ for all $\idm \in \Omega$. Since the $\widehat{R}_{\idm}$--modules  $N(\idm)$ and $\widehat{L}_{\idm}$ have the same rational envelope $\widehat{U}_{\idm}$, the $\widehat{R}_{\idm}$--module  $T(\idm):= \widehat{L}_{\idm}/N(\idm)$ has  finite length. Consequently, $T(\idm)$ has finite length viewed as an $R$--module and $\mathsf{Supp}\bigl(T(\idm)\bigr) = \{\idm\}$. We have a surjective homomorphism of $R$--modules $L \stackrel{c(\idm)}\lar 
T(\idm)$ given as the composition 
$
L \rightarrow \widehat{L}_{\idm} \dhrightarrow T(\idm).
$
Let $T = \oplus_{\idm \in \Omega} T(\idm)$. Then we get an $R$--module homomorphism  $L \stackrel{c}\lar T$, whose components are the maps $c(\idm)$ defined above. By construction, the map  $\widehat{L}_{\idm}  \stackrel{\hat{c}_{\idm}}\lar 
\widehat{T}_{\idm}$ is surjective for all $\idm \in \Max(R)$. As consequence,  the map $c$ is surjective, too. Let $N$ be the kernel of $c$. Then $N$ is a noetherian $R$--module, $K \cdot N = K\cdot L = U$  and  completions 
$\widehat{N}_{\idm} \subset \widehat{U}_{\idm}$
 are  given  by the formula  (\ref{E:prescribedcompletions}) for any $\idm \in \Max(R)$.

\smallskip 
\noindent
\underline{Case 2}. Consider now the general case, where $N(\idm) \subset \widehat{U}_{\idm}$ is an arbitrary $\widehat{R}_{\idm}$--lattice for $\idm \in \Omega$. Since $\widehat{K}_{\idm} \cdot \bigl(\widehat{L}_{\idm} + N(\idm)\bigr) = 
\widehat{K}_{\idm} \cdot \widehat{L}_{\idm}  = \widehat{U}_{\idm}$, the 
 $\widehat{R}_{\idm}$--module  
 $X(\idm):= \bigl(\widehat{L}_{\idm} + N(\idm)\bigr)/\widehat{L}_{\idm}$ has  finite length. It follows that  $X(\idm)$ has also finite length viewed as $R$--module and $\mathsf{Supp}\bigl(X(\idm)\bigr) = \{\idm\}$. Hence, there exists
$l \in \NN$ such that $\idm^l \cdot  X(\idm) = 0$. 

\smallskip
\noindent
Let $\bigl\{\idp_1, \dots, \idp_r\}$ be the set of minimal prime ideals of $R$. 
Then $D:= \idp_1 \cup \dots \cup \idp_r$ is the set of zero divisors of $R$. By prime avoidance, $\idm^l \not\subseteq D$, hence there exists $a^{(\idm)} \in \idm^l \setminus D$ such that $a^{(\idm)} X(\idm) = 0$, i.e.
$a^{(\idm)} N(\idm) \subseteq \widehat{L}_{\idm}$. Let $a := \prod_{\idm \in \Omega} a^{(\idm)}$. Then $a$ is a regular element in the ring $R$ and $a N(\idm) \subseteq \widehat{L}_{\idm}$ for any $\idm \in \Omega$. 

\smallskip
\noindent
Now we put $M:= \frac{1}{a} L \subset U$. Then $M$ is a lattice in $U$, $L \subseteq M$ and $N(\idm) \subseteq \widehat{L}_{\idm}$ for any $\idm \in \Omega$.
Let 
$
\Sigma :=
\bigl\{
\idm \in \Max(R) \, \big| \, a \in \idm\bigr\} = \mathsf{Supp}\bigl(R/(a)\bigr).
$
Then according to Case 1, there exists a sublattice $N \subseteq M$ such that
\begin{equation*}
\widehat{N}_{\idm} =
\left\{
\begin{array}{ll}
\widehat{M}_{\idm} & \mbox{\rm if} \;  \idm \notin \Sigma \cup \Omega \\
\widehat{L}_{\idm} & \mbox{\rm if} \;  \idm \in \Sigma \setminus \Omega \\
N(\idm) & \mbox{\rm if} \;  \idm \in \Omega.
\end{array}
\right.
\end{equation*}
This proves the existence of an $R$--lattice $N \subset U$ with the prescribed completions (\ref{E:prescribedcompletions}). 
The uniqueness of $N$ is a consequence of Lemma \ref{L:LatticeRationalLocal}. 
\end{proof}

\begin{remark}
The statement of  Theorem \ref{T:LocalModification} must be well--known to the experts. 
In the case when $R$ is an integral domain, it can be for instance  found  in  \cite[Th\'eor\`eme VII.4.3]{Bourbaki}. However,  we were not able to find a proof of  this result in the full generality  in the known literature. Since it plays a crucial role in our study of non--commutative curves, we decided to include a detailed proof for the sake of completeness and reader's convenience. 
\end{remark}

\medskip
\noindent
Let $\Lambda$ be a semi--simple $K$--algebra. Recall that a subring $A \subset \Lambda$ is an \emph{$R$--order}  if $R\cdot A = A$, $A$ is finitely generated  $R$--module and $K\cdot A = \Lambda$. Note that for any $\idm \in \Max(R)$ we have:
$$
\widehat{A}_{\idm} := \widehat{R}_{\idm} \otimes_R A \quad \mbox{\rm and} \quad
\widehat{\Lambda}_{\idm} := \widehat{K}_{\idm} \otimes_K \Lambda  \cong \widehat{R}_{\idm} \otimes_R \Lambda \cong \widehat{K}_{\idm} \otimes_R A.
$$
In particular, 
$\widehat{A}_{\idm}$ is an $\widehat{R}_{\idm}$--order in the semi--simple 
$\widehat{K}_{\idm}$--algebra
$\widehat{\Lambda}_{\idm}$.

\begin{proposition}\label{P:OrderPrescribedCompl} Let $\Lambda$ be a semi--simple 
$K$--algebra and $A \subset \Lambda$ be an $R$--order. 
Let $\Omega \subset \Max(R)$ be a finite subset such that for any $\idm \in \Omega$ we are given an $\widehat{R}_{\idm}$--order $B(\idm) \subset \widehat{\Lambda}_{\idm}$. Then there exists a unique $R$--order $B \subset \Lambda$ such that 
\begin{equation}\label{E:prescribedcompletionsorder}
\widehat{B}_{\idm} =
\left\{
\begin{array}{cc}
\widehat{A}_{\idm} & \mbox{\rm if} \;  \idm \notin \Omega \\
B(\idm) & \mbox{\rm if} \;  \idm \in \Omega.
\end{array}
\right.
\end{equation}
\end{proposition}

\begin{proof}
According to Theorem \ref{T:LocalModification} there exists a uniquely determined $R$--lattice $B \subset \Lambda$ with completions given by (\ref{E:prescribedcompletionsorder}). We have to show that $B$ us actually a subring. For any $\idm \in \Max(R)\setminus \Omega$ we put: $B(\idm) = \widehat{A}_{\idm}$. By Lemma \ref{L:LatticeRationalLocal} we have: 
$$
B = \left\{
b \in \Lambda \, \big| \, \hat{\vartheta}_{\idm}(b) \in B(\idm)
\; \mbox{\rm for all} \; \idm \in \Max(R)
\right\}.
$$
It follows that $b_1 + b_2, b_1 \cdot  b_2 \in B$ for all $b_1, b_2 \in B$. 
\end{proof}

\medskip
\noindent
Assume now that $U$ is a finitely generated left $\Lambda$--module. An $A$--submodule $L \subset U$, which is also an $R$--lattice,  is called \emph{$A$--lattice}. For any $\idm \in \Max(R)$ we have isomorphisms 
$$
\widehat{L}_{\idm} := \widehat{A}_{\idm} \otimes_A L \cong \widehat{R}_{\idm} \otimes_R L \quad \mbox{\rm and} \quad 
\widehat{U}_{\idm} := \widehat{\Lambda}_{\idm} \otimes_\Lambda U \cong \widehat{K}_{\idm} \otimes_K U \cong \widehat{R}_{\idm} \otimes_R U.
$$
It follows that $\widehat{L}_{\idm}$ is an $\widehat{A}_{\idm}$--lattice in 
$\widehat{U}_{\idm}$. 

\begin{proposition}\label{P:OrderPrescribedCompllattice} Let $U$ be a finitely generated $\Lambda$--module and $L\subset U$ an $A$--lattice. 
Let $\Omega \subset \Max(R)$ be a finite subset such that for any $\idm \in \Omega$ we are given an $\widehat{A}_{\idm}$--lattice $N(\idm) \subset \widehat{U}_{\idm}$. Then there exists a unique $A$--lattice $N \subset U$ such that 
\begin{equation}\label{E:prescribedcompletionslattice}
\widehat{N}_{\idm} =
\left\{
\begin{array}{cc}
\widehat{L}_{\idm} & \mbox{\rm if} \;  \idm \notin \Omega \\
N(\idm) & \mbox{\rm if} \;  \idm \in \Omega.
\end{array}
\right.
\end{equation}
\end{proposition}

\noindent
The proof of this result is the same as of  Proposition \ref{P:OrderPrescribedCompl}.

\section{Morita theorem for one--dimensional orders}

\noindent
As in the previous section, let $R$ be a reduced excellent ring of Krull dimension one, whose total ring of fractions is $K$. 
Recall that  one can also define the notion of an $R$--order without fixing  its rational envelope first. Namely, a finite $R$--algebra $A$ is an $R$--order if and only if it  is torsion free, viewed as as $R$--module and the ring  $\Lambda:= K \otimes_R A$ (the rational envelope of $A$) is semi--simple. Note that $A$ is an $R$--order if and only if $A$ is an $Z(A)$--order. 
If $R = Z(A)$ then $A$ is a \emph{central} $R$--order. If we have a ring extension $A \subseteq A'$ such that $A'$ is an $R$--order and $K \otimes_R A \rightarrow K \otimes_R A'$ is an isomorphism then  $A'$ is called \emph{overorder} of $A$. An order without proper overorders is called \emph{maximal}. 

\smallskip
\noindent
Analogously, a finitely generated (left) $A$--module $L$ is a (left) $A$--lattice if $L$ is torsion free viewed as an $R$--module. In this case, the $\Lambda$--module $V := K \otimes_R L$ is the rational envelope of $L$. If we have an extension of $A$--modules $L \subseteq N$ such that $L, N$ are both $A$--lattices and  the induced map
$K \otimes_R L \rightarrow K \otimes_R N$ is an isomorphism, then $L$ and $N$ are \emph{rationally equivalent} and $N$ is  an \emph{overlattice} of $L$ and $L$ is a \emph{sublattice} of $N$, respectively.

\subsection{Categorical characterization of the non--regular locus of an order} Let $\Lambda$ be a semi--simple $K$--algebra and $A \subset \Lambda$ be an $R$--order. Then the  set 
\begin{equation}
\gS_{A} := \bigl\{\idm \in \Max(R) \, \big|\, A_{\idm} \subset \Lambda_{\idm} \; \mbox{is not a maximal order}\bigr\} \subset \Spec(R)
\end{equation}
is the locus of   \emph{non--regular points} of $A$.

\begin{lemma}\label{L:NonRegLocus} The set $\gS_{A}$ is finite.
\end{lemma}
\begin{proof}
Let $\bigl\{\idp_1, \dots \idp_r\bigr\}$ be the set of minimal prime ideals in $R$, 
$R_i := R/{\idp}_i$ and $K_i := \Quot(R_i)$ for $1 \le i \le r$. Then we have injective ring homomorphisms: 
$$
R \xhar R':= R_1 \times \dots \times R_r \xhar
K_1 \times \dots \times K_r \cong K.
$$
In these terms, we get a decomposition: $\Lambda \cong \Lambda_1 \times \dots \times \Lambda_r$, where $\Lambda_i$ is a finite dimensional simple $K_i$--algebra for all $1 \le i \le r$. Let $A' := R' \cdot A \subset \Lambda$. Then $A'$ is an overorder of $A$ and we have a decomposition $A' \cong A'_1 \times \dots \times A'_r$, where $A'_i$ is an order in the simple algebra $\Lambda_i$. 
Let $\widetilde{R}_i$ be the integral closure of $R_i$ in $K_i$. Since $R$ is \emph{excellent} of Krull dimension one, the ring $\widetilde{R}_i$ is regular and the ring extension $R_i \subseteq \widetilde{R}_i$ is finite. It follows that  $A''_i :=
\widetilde{R}_i \cdot A'_i$ is an $\widetilde{R}_i$--order in the simple $K_i$--algebra
$\Lambda_i$. According to \cite[Corollary 10.4]{ReinerMO}, $A''_i$ is contained in a maximal order $\widetilde{A}_i$. Let $\widetilde{A} := \widetilde{A}_1 \times \dots \times \widetilde{A}_r$. Then $\widetilde{A}$ is a maximal order in $\Lambda$, which is  an  overorder of $A$. It follows from results of \cite[Section 11]{ReinerMO} that $\mathfrak{S}_A$ is the support of the finite length $R$--module $\widetilde{A}/A$. Hence, $\mathfrak{S}_A$ is a finite set. 
\end{proof}

\begin{remark}
Let $\idm \in \Max(R)$ be a \emph{regular} point of $A$, i.e.~$A_{\idm} \subset \Lambda_{\idm}$ is a maximal order. 
According to 
\cite[Lemma 2.3]{Harada}, its center $Z(A_{\idm})$ is a Dedekind ring. 
Moreover, $A_{\idm}$ itself is hereditary, too; see \cite[Theorem 18.1]{ReinerMO}.
 Note also that $A_{\idm}$ is a maximal order if and only if $\widehat{A}_{\idm}$ is a maximal order; see
 \cite[Theorem 11.5]{ReinerMO}. 
\end{remark}

\begin{lemma}\label{L:OverorderTechnical}
Let $A$ be an order in $\Lambda$ and $B$ its overorder such  that $A\cong B$, viewed as left $A$--modules. Then we have: $A = B$. 
\end{lemma}

\begin{proof}
First note that the  canonical morphism 
$$
\bigl\{\lambda \in A \,|\, B \lambda \subseteq A\bigr\} \lar \Hom_A(B, A), \; 
\lambda \mapsto (b \stackrel{\rho_\lambda}\mapsto b \lambda)
$$
is an isomorphism. Next, the following diagram is commutative:
$$
\xymatrix{
A^\circ \ar[rr]^-{\can} \ar@{_{(}->}[dd] & & \End_A(A) \\
& & \End_A(B) \ar[u]_-{\Ad_{\lambda}} \\
B^\circ \ar[rr]^-{\can}  & & \End_B(B), \ar[u]_-{\can}
}
$$
where $\lambda \in A$ is such that $B \lambda = A$ and all canonical arrows are isomorphisms. It follows that $A = B$, as asserted. 
\end{proof}

\smallskip
\noindent
It turns out that the set $\gS_{A}$ admits the following characterization. 

\begin{theorem}\label{T:MaximalityCategorical} Let $A$ be a central $R$--order in the semi--simple $K$--algebra 
$\Lambda$. 
Let $S$ be a simple $A$--module, $F := \End_A(S)$ the corresponding skew field  and $\mathsf{Supp}_R(S) = \{\idm\}$. Then $A_{\idm}$ is a maximal order if and only if the following conditions are satisfied:
\begin{itemize}
\item The length of the left $F$--module $\Ext_A^1(S, S)$ is one.
\item For any simple $A$--module $T \not\cong S$ we have: $\Ext^1_A(S, T) = 0$. 
\end{itemize}
\end{theorem}

\begin{proof} If $A_{\idm}$ is a maximal order then there exists a unique simple
$A$--module $S$ supported at $\idm$; see for instance  \cite[Theorem 18.7]{ReinerMO}. Moreover,  we have an isomorphism of left $F$--modules 
$\Ext_A^1(S, S) \cong F$. If $T$ is a simple $A$--module such that $T \not\cong S$ then $\mathsf{Supp}(T) \ne \mathsf{Supp}(S)$ and $\Ext^1_A(S, T) = 0$. 

\medskip
\noindent
To prove the converse direction, we may without loss of generality assume $R$ to be local and complete. Consider the short exact sequence in $A-\mathsf{mod}$:
$$
0 \lar Q \lar P \stackrel{\pi}\lar S \lar 0,
$$
where $P$ is a projective cover of $S$ and $Q = \mathsf{rad}(P)$ its radical. Then $P$ is indecomposable and the following sequence of left $F$--modules is exact:
$$
0 \lar \End_A(S) \stackrel{\pi^\ast}\lar \Hom_A(P, S) \lar \Hom_A(Q, S) \lar \Ext^1_A(S, S) \lar 0.
$$
Let $Q' = \mathsf{rad}(Q)$. Since $\pi^\ast$ is an isomorphism, we get:
$
\Hom_A\bigl(Q/Q', S\bigr) \cong \Hom_A(Q, S) \cong F.
$
Let $T \not\cong S$ be a simple $A$--module. Then we get an exact sequence
$$
0 \lar \Hom_A(S, T) \lar \Hom_A(P, T) \lar \Hom_A(Q, T) \lar \Ext^1_A(S, T) \lar 0.
$$
Since $P$ is a projective cover of a simple module $S$, we have: $\Hom_A(P, T) = 0$. By assumption, 
$\Ext^1_A(S, T) = 0$, hence $\Hom_A(Q, T) = 0$, too.  It follows that $Q/Q' \cong S$. Hence, there exists  a surjective homomorphism  of $A$--modules 
$P \stackrel{\nu}\lar Q$. Next,  $\bar{P} = \ker(\nu)$ is an $A$--lattice. Since $Q$ is a sublattice of $P$, we conclude that $K \otimes_R Q \cong K \otimes P$ and $K \otimes_R \bar{P} = 0$. Hence, $\bar{P} = 0$ and $\nu$ is an isomorphism. 

\smallskip
\noindent
Using induction on the length, one can show now that 
for any sublattice $P' \subseteq P$  we have: $P' \cong P$.
Moreover, we claim that 
$\widetilde{P}:= \Lambda \otimes_A P \cong K \otimes_R P$ is an indecomposable $\Lambda$--module. Indeed, if $\widetilde{P} \cong \widetilde{P}_1 \oplus \widetilde{P}_2$  then there exist  $A$--sublattices $P_i \subset \widetilde{P}_i$ and $P_1 \oplus P_2$ is a sublattice of $P$. From what was proven above it follows that $P_1 \oplus P_2 \cong P$. However, $P$ is indecomposable, hence $P_1 = 0$ or $P_2 = 0$. Thus,  $\widetilde{P}_1 = 0$ or $\widetilde{P}_2 = 0$, implying the claim. 

\smallskip
\noindent
Since $\widetilde{P}$ is an indecomposable projective $\Lambda$--module,  for any $A$--submodule $0 \ne X \subseteq P$ holds: $P/X$ has finite length and $X \cong P$. It implies   that for any indecomposable projective $A$--module $U \not\cong P$ we have: $\Hom_A(U, P) = 0$. Since $Z(A) = R$ is local, the algebra $A$ is connected. Since its rational envelope $\Lambda$  is semi--simple, we conclude that  there exists an isomorphism of left $A$--modules $A \cong P^{\oplus n}$ for some $n \in \NN$.

\smallskip
\noindent
Let $L$ be an indecomposable $A$--lattice. Then there exists an injective homomorphism of $A$--modules $L \xhar A^{\oplus m} \cong P^{\oplus mn}$ for some $m \in \NN$, hence $\Hom_A(L, P) \ne 0$. From what was proven above 
it follows that $L \cong P$. 

\smallskip
\noindent
Assume now that $A \subseteq A' \subset \Lambda$ is an overorder. Then $A'$ is an $A$--lattice, rationally equivalent to $A$.  Hence, $A$ and $A'$ are isomorphic as left $A$--modules. Lemma \ref{L:OverorderTechnical} implies that $A' = A$. Hence, the order $A$ is maximal, as asserted. 
\end{proof}

\subsection{Morita equivalences of central orders}
We developed all necessary tools to prove the following result.

\begin{proposition}\label{P:CentralMoritaOrders}
Let $A$ and $B$ be two central $R$--orders, whose  rational envelopes are semi--simple central $K$--algebras $\Lambda$ and $\Gamma$, respectively. Then $A$ and $B$ are centrally Morita equivalent if and only if the following conditions are satisfied:
\begin{itemize}
\item $\Lambda$ and $\Gamma$ are centrally Morita equivalent; 
\item we have: $\mathfrak{S}_A = \mathfrak{S}_B$;
\item for any $\idm \in \mathfrak{S}_A$,  the $\widehat{R}_{\idm}$--orders $\widehat{A}_{\idm}$ and $\widehat{B}_{\idm}$ are centrally Morita equivalent. 
\end{itemize}
\end{proposition}

\begin{proof} Assume that $A$  and $B$ are centrally Morita equivalent. Theorem \ref{T:MaximalityCategorical} implies that $\mathfrak{S}_A = \mathfrak{S}_B$. 
Let $P$ be a Morita $(B-A)$--bimodule, for which the left and right actions of $R$ coincide. Then the following diagram of rings and ring homomorphisms 
\begin{equation}\label{E:CommDiag2}
\begin{array}{c}
\xymatrix{
\End_A(P) & &  B \ar[ll]_-{\lambda^P} \\
 & R \ar@{_{(}->}[lu]^-{\rho^P}  \ar@{^{(}->}[ru] & 
}
\end{array}
\end{equation}
is commutative. Passing  in (\ref{E:CommDiag2})  to localizations and completions, we conclude that 
$\widehat{A}_{\idm}$ and $\widehat{B}_{\idm}$ are centrally Morita equivalent for all $\idm \in \Max(R)$. In the same way, $\Lambda$ and $\Gamma$ are centrally Morita equivalent.

\smallskip
\noindent
Proof of the converse direction is more involved. Let $V$ be a Morita $(\Gamma-\Lambda)$--bimodule inducing a central  equivalence of categories  $\Lambda-\mathsf{mod} \lar 
\Gamma-\mathsf{mod}$. For any $\idm \in \gS := \gS_{A}$, we get 
a Morita $(\widehat{\Gamma}_{\idm}-\widehat{\Lambda}_{\idm})$--bimodule $\widehat{V}_{\idm}$, which induces  a central equivalence of categories  $\widehat{\Lambda}_{\idm}-\mathsf{mod} \lar \widehat{\Gamma}_{\idm}
-\mathsf{mod}$. Let $P(\idm)$ be a Morita $(\widehat{B}_{\idm}-\widehat{A}_{\idm})$--bimodule inducing a central  equivalence of categories $\widehat{A}_{\idm}-\mathsf{mod} \lar 
\widehat{B}_{\idm}-\mathsf{mod}$. Then $$\widetilde{P}(\idm) := \widehat{K}_{\idm}
\otimes_{\widehat{R}_{\idm}} P(\idm) \cong  \widehat{\Gamma}_{\idm}
\otimes_{\widehat{B}_{\idm}} P(\idm) \cong P(\idm) 
\otimes_{\widehat{A}_{\idm}} \widehat{\Lambda}_{\idm}
$$
is a Morita $(\widehat{\Gamma}_{\idm}-\widehat{\Lambda}_{\idm})$--bimodule, which induces  a central  equivalence of categories $\widehat{\Lambda}_{\idm}-\mathsf{mod} \rightarrow
\widehat{\Gamma}_{\idm}-\mathsf{mod}$
as well. 
Since $\widehat{\Lambda}_{\idm}$ and $\widehat{\Gamma}_{\idm}$ are semi--simple rings, 
Theorem \ref{T:SkolemNoether} implies that $\widehat{V}_{\idm}$ and $\widetilde{P}(\idm)$ are isomorphic as $(\widehat{\Gamma}_{\idm}-\widehat{\Lambda}_{\idm})$--bimodules. Therefore, we can without loss of generality assume that 
$P(\idm) \subset \widehat{V}_{\idm}$ and  the left action of 
$\widehat{B}_{\idm}$ as well as  the right action of 
$\widehat{A}_{\idm}$  on $P(\idm)$ and 
$\widehat{V}_{\idm}$
match for all $\idm \in \gS$. Our goal now is to construct a $(B-A)$--subbimodule
$P \subset V$, which induces a central equivalence of categories 
$A-\mathsf{mod} \rightarrow B-\mathsf{mod}$. 

\smallskip
\noindent
We start with an arbitrary   $R$--lattice $L \subset V$ and put: 
$Q := B \cdot L \cdot A \subset V$. Then $Q$ is a finitely generated $R$--module and $K \cdot Q = V$, i.e. $Q$ is an $R$--overlattice of $L$ with the same rational envelope  $V$. Moreover, $Q$ is a $(B-A)$--bimodule with central action of $R$. Note that  the following diagram of rings and ring homomorphisms
$$
\xymatrix{
B  \ar@{_{(}->}[d]   \ar[rr]^-{\lambda_B^Q}  & & \End_A(Q) 
\ar@{^{(}->}[d] \\
\Gamma \ar[rr]^-{\lambda_\Gamma^V} & & \End_\Lambda(V) 
}
$$
is commutative. Since  $\mu := \lambda_\Gamma^V$ is  an isomorphism, the map $\lambda := \lambda_B^Q$ is injective. Hence,  $ 
\widehat{B}_{\idm} \stackrel{\widehat{\lambda}_{\idm}}\lar 
\End_{\widehat{A}_{\idm}}(\widehat{Q}_{\idm})$ is injective for any $\idm \in \Max(R)$. 

\smallskip
\noindent
We claim  that $\widehat{\lambda}_{\idm}$ is  an isomorphism for all
$\idm \in \Max(R) \setminus \gS$. Indeed, $ 
\widehat{\Gamma}_{\idm} \stackrel{\widehat{\mu}_{\idm}}\lar 
\End_{\widehat{\Lambda}_{\idm}}(\widehat{V}_{\idm})$ is an isomorphism and 
$ 
\widehat{B}_{\idm}$ is a maximal order in the semi--simple algebra $ 
\widehat{\Gamma}_{\idm}$. Hence, 
$\widehat{\lambda}_{\idm}\bigl(\widehat{B}_{\idm} \bigr)$ is a maximal order in the semi--simple algebra $\End_{\widehat{\Lambda}_{\idm}}(\widehat{V}_{\idm})$ and as a consequence,  we get:
$\widehat{\lambda}_{\idm}\bigl(\widehat{B}_{\idm}\bigr) = \End_{\widehat{A}_{\idm}}(\widehat{Q}_{\idm})$.

\smallskip
\noindent
According with  Lemma \ref{L:NonRegLocus}, the set $\gS$ is finite. By Proposition
\ref{P:OrderPrescribedCompllattice},  there exists a unique right $A$--lattice $P \subset V$ such that 
$
\widehat{P}_{\idm} =
\left\{
\begin{array}{cl}
\widehat{Q}_{\idm} & \mbox{\rm if} \;  \idm \notin \gS \\
P(\idm) & \mbox{\rm if} \;  \idm \in \gS.
\end{array}
\right.
$

\noindent
It follows from Lemma \ref{L:LatticeRationalLocal} that $B \cdot P = P$, i.e.~ $P$ is an $(B-A)$--bimodule with central action of $R$. Since 
$\widehat{P}_{\idm}$ is a right $\widehat{A}_{\idm}$--progenerator 
for any $\idm \in \Max(R)$,  $P$ is a right $A$--progenerator. 
The ring homomorphism $B  \stackrel{\lambda_B^P}\lar  \End_A(P)$ is an isomorphism, since so are its completions for all $\idm \in \Max(R)$. Hence, $P$ is a Morita
$(B-A)$--bimodule, which induces a central equivalence of categories  $A-\mathsf{mod} \lar  
B-\mathsf{mod}$ we are looking for. 
\end{proof}

\begin{lemma}\label{L:ConjugateMorita}
Let $\Lambda$ be a semi--simple central $K$--algebra, $A \subset \Lambda$ be a central $R$--order, $\upsilon \in \Lambda^\ast$ and $B:= \upsilon A \upsilon^{-1}$. Then $A$ and $B$ are centrally Morita equivalent.
\end{lemma}

\begin{proof}
Let $P:= \upsilon A$. Obviously, $P \cong A$ as right $A$--modules. In particular, $P$ is a right $A$--progenerator. Note that $B \cdot P = 
\bigl(\upsilon A \upsilon^{-1}\bigr)\cdot (\upsilon A) = \upsilon A = P$. It follows that $P$ is an $(B-A)$--bimodule with central $R$--action.

\noindent
Since $\upsilon \in \Lambda^\ast$, the ring homomorphism 
$B \stackrel{\lambda^P_B}\lar \End_A(P)$ is injective. Let $f \in \End_A(P)$.
As  $P$ is a right $A$--lattice, whose rational envelope is  $\Lambda_{\Lambda}$, there exists $b \in \Lambda$ such that $f = \lambda^P_b$.  Since 
$b \upsilon A \subseteq \upsilon A$, there exists $a \in A$ such that $b \upsilon = \upsilon a$. It follows that $b = \upsilon a \upsilon^{-1} \in B$. Hence, $\lambda^P_B$ is an isomorphism and the $(B-A)$--bimodule $P$ induces a central equivalence we are looking for. 
\end{proof}

\subsection{Morita equivalences of non--commutative curves}

\noindent
A \emph{reduced non--commutative curve} (abbreviated as \emph{rncc}) is a ringed space $\bbX = (X, \kA)$, where $X$ is a reduced excellent noetherian scheme of pure dimension one and $\kA$ is a sheaf of $\kO$--orders (i.e $A(U)$ is an $O(U)$--order for any affine open subset $U \subseteq X$). Following Definition \ref{D:NCNS}, we say that $\bbX$ is \emph{central} if 
$O_x = Z(A_x)$ for all $x \in X$.

\smallskip
\noindent
From now on, let $\bbX$ be a central rncc. If $\kK$ is the sheaf of rational functions on $X$ then $K := \Gamma(X, \kK)$ is a semi--simple ring. Moreover, $\Lambda_{\bbX}:= \Gamma(X, \kK \otimes_\kO \kA)$ is a central semi--simple $K$--algebra. This algebra can be viewed as the ``ring of rational functions'' on  $\bbX$. If $\Coh_0(\bbX)$ denotes the abelian category of 
objects of finite length in $\Coh(\bbX)$ then the functor 
$\Gamma(X, \kK \otimes_\kO - )$ induces an equivalence of categories
$
\Coh(\bbX)/\Coh_0(\bbX) \simeq \Lambda_{\bbX}-\mathsf{mod}.
$

\smallskip
\noindent
For any $x \in X$, we have a central $\widehat{O}_x$--order $\widehat{A}_x$, whose rational envelope can be canonically identified with the  semi--simple ring
$\widehat{\Lambda}_x = \widehat{K}_x \otimes_K \Lambda_{\bbX}$. Let
\begin{equation}
\gS_{\bbX} := \{x \in X \, \big|\, \widehat{A}_x 
 \; \mbox{is not a maximal order in} \; 
\widehat{\Lambda}_x\}
\end{equation}
be the locus of \emph{non--regular} points of $\bbX$. According to Lemma \ref{L:NonRegLocus}, $\gS_{\bbX}$ is a finite set.  Let 
$x \in X$ be a \emph{regular}  point. By  \cite[Lemma 2.3]{Harada}, 
$\widehat{O}_x = Z\bigl(\widehat{\Lambda}_x\bigr)$ is a discrete valuation ring. 
It follows that $\widehat{K}_x$ is a field and $\widehat{\Lambda}_x$ is a central simple 
$\widehat{K}_x$--algebra. Hence, there exists a skew field $F_x \supseteq \widehat{K}_x$ (such that $Z(F_x) = \widehat{K}_x$) and $n = n(x) \in \NN$ such that $\widehat{\Lambda}_x \cong \Mat_n(F_x)$.
Moreover,  one can explicitly describe the order
$\widehat{A}_x$ in this case. 
Namely, there  exists a uniquely determined maximal $\widehat{R}_x$--order $T_x \subset F_x$; see \cite[Theorem 12.8]{ReinerMO}, and  $\Mat_n(T_x)$ can be identified with a maximal order in $\widehat{\Lambda}_x$. Moreover, any two maximal orders $O_x', O_x'' \subset \widehat{\Lambda}_x$ are conjugate, i.e.~there exists $\upsilon \in \widehat{\Lambda}_x^\ast$ such that $O_x'= \upsilon 
O_x'' \upsilon^{-1}$; 
 see \cite[Theorem 17.3]{ReinerMO}. In particular, the orders $O_x'$ and $O_x''$ are centrally Morita equivalent (see Lemma \ref{L:ConjugateMorita}) and 
 $\widehat{A}_x \cong \Mat_n(T_x)$.

\begin{proposition}\label{P:CentralMoritaEqNCC}
Let $\bbX = (X, \kA)$ and $\bbY = (X, \kB)$ be two central rncc with the same central curve
$X$. Then the categories  $\Qcoh(\bbX)$ and 
$\Qcoh(\bbY)$ are \emph{centrally} equivalent if any only if  the following conditions are satisfied:
\begin{itemize}
\item the semi--simple $K$--algebras $\Lambda_{\bbX}$ and $\Lambda_{\bbY}$ are centrally Morita equivalent; 
\item we have: $\mathfrak{S}_{\bbX} = \mathfrak{S}_{\bbY}$;
\item for any $x \in \mathfrak{S}_{\bbX}$,  the $\widehat{O}_{x}$--orders $\widehat{A}_{x}$ and $\widehat{B}_{x}$ are centrally Morita equivalent. 
\end{itemize}
\end{proposition}
\begin{proof} This result is just a global version of Proposition \ref{P:CentralMoritaOrders}
and the proof below is basically a ``sheafified'' version  of the arguments  from  the affine case. 
 
\smallskip
\noindent 
 Let $\Qcoh(\bbX) \stackrel{\Phi}\lar
\Qcoh(\bbY)$ be a central equivalence of categories. 
According to Theorem \ref{T:MoritaNonCommCurves}, we have:  $\Phi \cong \kP \otimes_{\kA} \,-\,$, where 
$\kP$ is a sheaf of $(\kB-\kA)$--bimodules such that 
$\kP \in \VB(\bbX^\circ)$ is a local right progenerator and 
$\kB \stackrel{\lambda_\kB^\kP}\lar \bigl(\mathit{End}_{\kA}(\kP)\bigr)$ an isomorphism of $\kO$--algebras. Let $V:= \Gamma(X, \kK \otimes \kP)$. Then $V$ is a Morita 
$(\Lambda_{\bbY}-\Lambda_{\bbX})$--bimodule inducing a central   Morita equivalence 
$\xymatrix{\Lambda_{\bbX} \ar@{.>}[r] & \Lambda_{\bbY}}$. 
Similarly, for any $x \in X$ we get a central Morita $\bigl(\widehat{B}_x - \widehat{A}_{x}\bigr)$--bimodule
$\widehat{P}_{x}$.  Finally, 
Theorem \ref{T:MaximalityCategorical} implies that  $\mathfrak{S}_{\bbX} = \mathfrak{S}_{\bbY}$.

\smallskip
\noindent
Conversely, assume that rncc $\bbX$ and $\bbY$ satisfy three conditions above. Let $V$ be a Morita
$(\Lambda_{\bbY}-\Lambda_{\bbX})$--bimodule with central action of $K$. 
Let $\widetilde\kA = \kK \otimes_\kO \kA$ and $\widetilde\kB = \kK \otimes_\kO \kB$.
Passing to sheaves, we get 
a balanced $(\widetilde{\kB}-\widetilde{\kA})$--bimodule $\kV$ with central action of $\kK$. Let $\kL \subset \kV$ be any sheaf of  $\kO$--lattices, i.e.~a coherent submodule of $\kV$ such that the canonical morphism $\kK \otimes_\kO \kL \rightarrow \kV$ is an isomorphism. 
Then  $\kQ := \kB \cdot \kL \cdot \kA \subset \kV$  is a sheaf of lattices with rational envelope $\kV$. Moreover, $\kQ$  is  a  $({\kB}-{\kA})$--bimodule having a  central action of $\kO$. As in the proof of Proposition \ref{P:CentralMoritaOrders}, one can  show that for any $x \in \mathfrak{S} := \mathfrak{S}_{\bbX}$ there exists  a Morita $(\widehat{B}_x-\widehat{A}_x)$--bimodule $P(x) \subset \widehat{V}_x$, having a  central action of $\widehat{O}_x$. By Theorem \ref{T:LocalModification}, there exists a unique $\kP \in \VB(\bbX^\circ)$ such that 
$\kP \subset \kV$ is a right $\kA$--lattice and 
\begin{equation*}
\widehat{P}_{x} =
\left\{
\begin{array}{cl}
\widehat{Q}_{x} & \mbox{\rm if} \;  \idm \notin \gS \\
P(x) & \mbox{\rm if} \;  x \in \gS.
\end{array}
\right.
\end{equation*}
Then $\kP$ is a sheaf of  $(\kB-\kA)$--bimodules with  central action of $\kO$. Moreover,  $\kP \in \VB(\bbX^\circ)$ is a local right progenerator such that 
 $\kB  \stackrel{\lambda_\kB^\kP}\lar  \mathit{End}_{\kA}(\kP)$ is an isomorphism of $\kO$--algebras. Hence, $\Phi:= \kP \otimes_\kA \,-\,$ is a  central equivalence of categories  we are looking for. 
\end{proof}

\smallskip
\noindent
After all preparations  we can now prove the main result of this section.

\begin{theorem}\label{T:MoritaNonCommCurves}
Let $\bbX = (X, \kA)$ and $\bbY = (Y, \kB)$ be two central rncc. Then the categories  $\Qcoh(\bbX)$ and 
$\Qcoh(\bbY)$ are equivalent if any only if there exists a scheme isomorphism
$Y \stackrel{\varphi}\lar X$ satisfying the following conditions.
\begin{itemize}
\item There exists a Morita equivalence $\xymatrix{\Lambda_{\bbX} 
\ar@{.>}[r]^-{\widetilde\Phi} & \Lambda_{\bbY}}$ such  that the diagram 
\begin{equation}\label{E:MoritaBirational}
\begin{array}{c}
\xymatrix{
 \Lambda_{\bbX} \ar@{.>}[r]^-{\widetilde\Phi} & \Lambda_{\bbY}  \\
K_X \ar[r]^-{\varphi^\ast} \ar@{^{(}->}[u] & K_Y \ar@{_{(}->}[u]
}
\end{array}
\end{equation}
is ``commutative'' (here we follow the notation of Remark \ref{R:Pseudocommutative}).
\item We have:  $\varphi\bigl(\gS_{\bbY}\bigr) = \gS_{\bbX}$ and for any $y \in \gS_{\bbY}$, there exists a Morita equivalence 
$\xymatrix{\widehat{A}_{\varphi(y)} \ar@{.>}[r]^-{\Phi_y} & \widehat{B}_{y}}$ 
such  that the diagram 
\begin{equation}\label{E:MoritaComplete}
\begin{array}{c}
\xymatrix{
 \widehat{A}_{\varphi(y)}  \ar@{.>}[r]^-{\Phi_y} &  \widehat{B}_{y} \\
\widehat{O}_{\varphi(y)} \ar[r]^-{\varphi_y^\ast} \ar@{^{(}->}[u] &  
\widehat{O}_{y} \ar@{_{(}->}[u]
}
\end{array}
\end{equation}
is ``commutative''.
\end{itemize}
\end{theorem}

\begin{proof} According to Theorem \ref{T:MoritaNonCommCurves}, any equivalence of categories 
 $\Qcoh(\bbX) \stackrel{\Phi}\lar
\Qcoh(\bbY)$ is isomorphic to a  functor of the form $\Phi_{\kP, \vartheta, \varphi}$, where 
$Y \stackrel{\varphi}\lar X$ is a scheme isomorphism, 
$\kP \in \VB(\bbX^\circ)$ a local right progenerator and 
$\kB \stackrel{\vartheta}\lar 
\varphi^\ast\bigl(\mathit{End}_{\kA}(\kP)\bigr)$ an isomorphism of $\kO_Y$--algebras. Let $V:= \Gamma(X, \kK_X \otimes \kP)$. Then $V$ is a Morita 
$(\Lambda_{\bbY}-\Lambda_{\bbX})$--bimodule inducing an  equivalence 
$\xymatrix{\Lambda_{\bbX} \ar@{.>}[r]^-{\widetilde\Phi} & \Lambda_{\bbY}}$ such  that the diagram (\ref{E:MoritaBirational}) is ``commutative''. Similarly, for any $y \in Y$ we get a Morita $\bigl(\widehat{B}_y - \widehat{A}_{\varphi(y)}\bigr)$--bimodule
$\widehat{P}_{\varphi(y)}$, which induces  an equivalence $\xymatrix{\widehat{A}_{\varphi(y)} \ar@{.>}[r]^-{\Phi_y} & \widehat{B}_{y}}$ 
such  that the diagram (\ref{E:MoritaComplete}) is ``commutative''. Finally, 
Theorem \ref{T:MaximalityCategorical} implies that $y \in \gS_{\bbY}$ if and only if $\varphi(y) \in \gS_{\bbX}$.

\smallskip
\noindent
Conversely, assume that we are given a scheme isomorphism $Y \stackrel{\varphi}\lar X$ as well as Morita equivalences $\xymatrix{\Lambda_{\bbX} \ar@{.>}[r]^-{\widetilde\Phi} & \Lambda_{\bbY}}$ and $\bigl(\xymatrix{\widehat{A}_{\varphi(y)} \ar@{.>}[r]^-{\Phi_y} & \widehat{B}_{y}}\bigr)_{y \in \gS_{\bbY}}$ satisfying the compatibility constraints (\ref{E:MoritaBirational}) and (\ref{E:MoritaComplete}). Let   $\bbY' = \bigl(Y, \varphi^\ast(\kA)\bigr)$. Then 
$\bbY'$ is a central  rncc  and $\gS_{\bbY} = \gS_{\bbY'}$. By Proposition \ref{P:CentralMoritaEqNCC},
 $\bbY$ and $\bbY'$ are centrally Morita equivalent rncc, implying the result.
\end{proof}

\subsection{Morita equivalences of hereditary non--commutative curves}
Recall that   a  central rncc  $\XX = (X, \kA)$ is \emph{hereditary} if 
$\widehat{A}_x$ is a hereditary order for all $x \in X$. In this case, the central curve $X$ is automatically regular; see  \cite[Theorem 2.6]{Harada}. We may without loss of generality assume  $X$ to be connected, hence $\Lambda_{\bbX}$ is a central simple $K_X$--algebra, which defines an element in  the Brauer group of the function field $K_X$. 

\smallskip
\noindent
Following the notation of the previous subsection, for any $x \in X$ there exists
a skew field $F_x$ (whose center is $\widehat{K}_x$) and $n = n(x) \in \NN$ such that $\widehat{\Lambda}_x \cong \Mat_n(F_x)$. Any maximal order in $\widehat{\Lambda}_x$ is conjugate to $\Mat_n(T_x)$, where $T_x$ is the unique maximal order in $F_x$.

\smallskip
\noindent
Let  $x \in \gS_{\XX}$ and $\mathfrak{t}_x$ be the Jacobson radical of $T_x$. Viewing $\widehat{H}_x$ as an order in the simple algebra $\Mat_n(F_x)$, we have the following result:
$\widehat{A}_x$ is \emph{conjugate} to the order
\begin{equation}\label{E:StandardHereditary}
H\bigl(T_x, (n_1, \dots, n_t)\bigr) :=
\left[
\begin{array}{cccc}
T_x & \mathfrak{t}_x & \dots & \mathfrak{t}_x\\
T_x & T_x & \dots & \mathfrak{t}_x\\
\vdots & \vdots & \ddots & \vdots \\
T_x & \mathfrak{t}_x & \dots & T_x\\
\end{array}
\right]^{\underline{(n_1, \dots, n_t)}} \subseteq \Mat_n(T_x)
\end{equation}
where $(n_1,  \dots,  n_t) \in \NN^t$ is  such that
$n = n_1 + \dots + n_t$. The elements of $H\bigl(T_x, (n_1, \dots, n_t)\bigr)$ are matrices, such that for any $1 \le i, j \le t$, the $(i,j)$-th entry is itself an arbitrary  matrix of size $(n_i \times n_j)$ with coefficients in $T_x$ for $i \ge j$ and  in  $\mathfrak{t}_x$ for $i < j$. The length of this tuple $t = t(x)$ (called \emph{type} of the hereditary order $\widehat{H}_x$) is equal to   the number of non--isomorphic simple $\widehat{H}_x$--modules. Moreover, $H\bigl(T_x, (n_1, \dots, n_t)\bigr) \cong H\bigl(T_x, (n'_1, \dots, n'_t)\bigr)$ if and only if 
$(n_1, \dots, n_t)$ and $(n'_1, \dots, n'_t)$ are obtained from each other by a cyclic permutation. 
 We refer to  \cite[Theorem 39.14 and Corollary 39.24]{ReinerMO}  for proofs  of all these  fundamental facts. 

\smallskip
\noindent
A point $x \in X$ is a regular point of $\XX$ if and only if $t(x) = 1$. 
It is easy to see that   $H\bigl(T_x, (n_1, \dots, n_t)\bigr)$ is centrally Morita equivalent to the basic  order $H\bigl(T_x, \underbrace{(1, \dots, 1)\bigr)}_{\scriptstyle t \; \mathrm{times}}$.
It follows from Lemma \ref{L:ConjugateMorita}
that any two hereditary orders of the same type in the simple algebra  $\widehat{\Lambda}_x$  are \emph{centrally} Morita equivalent. 
Theorem \ref{T:MoritaNonCommCurves} implies the following result, which was proven for the first time by Spie\ss{}; see \cite[Proposition 2.9]{Spiess}.

\begin{corollary}\label{C:HereditaryMorita} Let $\bbX = (X, \kA)$ and $\bbY = (Y, \kB)$ be two hereditary rnnc  and $\Lambda_{\bbX}$ and $\Lambda_{\bbY}$ the corresponding simple algebras over the function fields $K_X$ and $K_Y$, respectively. Then the categories $\Qcoh(\bbX)$ and $\Qcoh(\bbY)$
are equivalent if and only if there exists an isomorphism  $Y \stackrel{\varphi}\lar X$ such that  $t(\varphi(y)) = t(y)$ for any $y \in Y$ and 
$\bigl[\Lambda_{\bbY}\bigr] = \bigl[\varphi^*(\Lambda_{\bbX})\bigr] \in \mathsf{Br}(K_Y)$, where $\mathsf{Br}(K_Y)$ is the Brauer group of $K_Y$.
\end{corollary}

\begin{remark} In the setting  of Corollary \ref{C:HereditaryMorita},
assume additionally that $X$ and $Y$ are quasi--projective curves over an
algebraically closed field $\kk$. By  Tsen's theorem, we have: 
$\mathsf{Br}(K_X) = 0  = \mathsf{Br}(K_Y)$; see \cite[Proposition 6.2.3 and  Theorem 6.2.8]{GilleSzamuely}. It follows that $\bbX$ and $\bbY$ are Morita equivalent if and only if there exists an isomorphism  $Y \stackrel{\varphi}\lar X$ such that  $t\bigl(\varphi(y)\bigr) = t(y)$ for any $y \in Y$.
\end{remark}

\begin{remark}
A hereditary rncc  $\bbX = (X, \kA)$ is called \emph{regular} if 
$\gS_{\bbX} = \emptyset$. Assume additionally that $X$ is a \emph{projective} curve over a field $\kk$ (as already mentioned, the central curve $X$ is automatically regular in this case). Then $\Coh(\bbX)$ is a noetherian hereditary category with finite dimensional $\Hom$-- and 
$\Ext$--spaces, admitting an Auslander--Reiten translation functor $\Coh(\bbX)\stackrel{\tau}\lar \Coh(\bbX)$ such that $\tau(\kF) \cong \kF$ for any object $\kF$ of $\Coh_0(\bbX)$. Various properties of the  category $\Coh(\bbX)$  were  studied in detail (from a  slightly different perspective) by Kussin in \cite{Kussin}. 

\smallskip
\noindent
Let $\Lambda = \Lambda_{\bbX}$ be the algebra of ``rational functions'' on $\bbX$ and $K = K_X$ be its center. There exists a unique (up an isomorphism)  regular projective curve over $\kk$ (namely,  $X$ itself), whose field of rational functions is isomorphic to $K$; see for instance \cite[Proposition  7.3.13]{Liu}. Let $\kB$ be any sheaf of \emph{maximal} orders on $X$ such that $\Gamma(X, \kK \otimes_\kO \kB) \cong \Lambda$. It is well--known that the ringed spaces $\bbX = (X, \kA)$ and $\bbX' = (X, \kB)$ need not be in general isomorphic (see for instance \cite{Deuring} or \cite[Corollary]{DrozdTurchin} for examples of non--isomorphic maximal orders in the same central simple algebra). However, Corollary \ref{C:HereditaryMorita} implies that the 
categories $\Qcoh(\bbX)$ and $\Qcoh(\bbX')$ are equivalent.  Hence,  a regular rncc  $\bbX$ with central projective curve $X$ is up to a Morita equivalence determined by an element in the Brauer group 
$\mathsf{Br}(K_X)$. \qed
\end{remark}

\smallskip
\noindent
The following example shows that the compatibility constraints (\ref{E:MoritaBirational}) and (\ref{E:MoritaComplete}) are necessary to end up with a global Morita equivalence. 

\begin{example}\label{E:ExampleNonHered}
Let $\kk$ be an infinite field $S = \kk[x]$, $J = \bigl((x-\lambda')(x-\lambda'')(x^2-1)\bigr)$ for $\lambda' \ne \lambda''  \in \kk \setminus \{1, -1\}$ and 
$
H = \left(
\begin{array}{ccc}
S & J & J \\
S & S & J \\
S & S & S
\end{array}
\right).
$
Next, we put:
$$
A_+ = \left\{p  \in H \left|  \begin{array}{l}
 p_{11}(1) = p_{11}(-1) \\
 p_{22}(1) = p_{33}(1)
 \end{array}
 \right.
 \right\} \quad \mbox{\rm and} \quad
A_- = \left\{p  \in H \left|  \begin{array}{l}
 p_{11}(1) = p_{11}(-1) \\
 p_{22}(-1) = p_{33}(-1)
 \end{array}
 \right.
 \right\}. 
$$
Then we have: $R := Z\bigl(A_{\pm}\bigr) = \kk\bigl[x^2-1, x(x^2-1)\bigr] \cong 
\kk[u, v]/(f),
$
where $f = v^2-u^3-u^2$. It is clear that $A_\pm$ are $R$--orders with common rational envelope $\Lambda = \Mat_3\bigl(\kk(x)\bigr)$. Passing to the corresponding sheaves of orders,  we get a pair of  rncc  $\bbE_\pm := (E, \kA_\pm)$, where $E = V(f) \subset \mathbbm{A}^2$ is a plane nodal cubic. These curves have the same locus of non--regular points $\gS = \{s, q', q''\}$, where $s = (0, 0)$ is the singular point of $E$, whereas  $q' = \bigl(\lambda'^2-1, 
\lambda' (\lambda'^2-1)\bigr)$ and $q'' = \bigl(\lambda''^2-1, 
\lambda'' (\lambda''^2-1)\bigr)$  are two distinct  regular points of $E$. 
We have:
$$
\bigl(\widehat{A}_+\bigr)_{q'} = \bigl(\widehat{A}_-\bigr)_{q'} = \left(
\begin{array}{ccc}
O' & \idm' & \idm' \\
O' & O' & \idm' \\
O' & O' & O'
\end{array}
\right),
$$
where $O' = \widehat{O}_{q'}$ is the completed local ring of $E$ at the point $q'$ and $\idm'$ its maximal ideal. Of course, the analogous statement holds for the second point $q''$, too. 

\smallskip
\noindent
Both orders $\bigl(\widehat{A}_+\bigr)_s$ and $\bigl(\widehat{A}_-\bigr)_s$
have the common center $D = \kk\llbracket w_+, w_-\rrbracket/(w_+ w_-)$, where
$w_\pm = v \pm u(1 - u + u^2 - \dots) \in \kk\llbracket u, v\rrbracket$.  It is clear that $\bigl(\widehat{A}_+\bigr)_s$ and $\bigl(\widehat{A}_-\bigr)_s$ are  isomorphic as rings. 
However, although $\bbE_+$ and $\bbE_-$ have the same central curve $E$, the common algebra of rational functions $\Lambda$ and the same singularity types, they are \emph{not Morita equivalent}!

\smallskip
\noindent
Indeed, any equivalence of categories $\Qcoh(\bbE_+) \stackrel{\Phi}\lar \Qcoh(\bbE_-)$ induces  an automorphism $E \stackrel{\varphi}\lar E$; see 
Theorem \ref{T:MorphismCentralSchemes}. It follows that $\varphi(s) = s$ and 
$\varphi\bigl(\left\{q', q''\right\}\bigr) = \left\{q', q''\right\}$.
In particular, we get 
the restricted  Morita equivalence 
$\xymatrix{\bigl(\widehat{A}_+\bigr)_s  \ar@{.>}[r]^{\Phi_s} &  \bigl(\widehat{A}_-\bigr)_s}$.
However, any such equivalelnce induces  an automorphism of $D$ swapping both branches $(w_+)$ and $(w_-)$. If $\lambda', \lambda''$ were  chosen  sufficiently general, such an automorphism $\varphi$ does not exist. Hence, 
$\Qcoh(\bbE_+)$ and $\Qcoh(\bbE_-)$ are not equivalent, as asserted. \qed
\end{example}


\begin{thebibliography}{99}

\bibitem{Antieau}
B.~Antieau, \emph{A reconstruction theorem for abelian categories of twisted sheaves}, J. Reine Angew. Math. \textbf{712} (2016), 175--188.

\bibitem{ArtinZhang}
M.~Artin \& J.~Zhang, \emph{Noncommutative projective schemes}, 
 Adv. Math. \textbf{109} (1994), no. 2, 228--287.



\bibitem{Bass}
H.~Bass, \emph{Algebraic K--theory},  W. A. Benjamin, Inc., New York--Amsterdam 1968.







\bibitem{Bourbaki}
 N.~Bourbaki, \emph{\'El\'ements de math\'ematique. Alg\`ebre commutative. 
 Chapitres 5 \`a 7},  Masson, Paris, 1985. 351 pp.






\bibitem{bd}
 I.~Burban \& Yu.~Drozd,
 \emph{Tilting on non--commutative rational projective curves},
 Math.~Ann.~\textbf{351}, no.~3 (2011), 665--709.
 
 \bibitem{bdnpdalcurves} I.~Burban \& Yu.~Drozd,
 \emph{Non--commutative nodal curves and derived tame algebras}, 
 \texttt{arXiv:1805.05174}.



\bibitem{BurbanDrozdGavranIMRN}
I.~Burban, Yu.~Drozd \& V.~Gavran, \emph{Singular curves and quasi--hereditary algebras}, Int.~Math.~Res.~Notices \textbf{2017}, no.~3, 895--920 (2017).



\bibitem{CaldararuThesis}
 A.~C\u{a}ld\u{a}raru, \emph{Derived categories of twisted sheaves on Calabi--Yau manifolds},  Thesis (Ph.D.)–Cornell University. 2000. 196 pp

\bibitem{CaldararuCrelle}
  A.~C\u{a}ld\u{a}raru,  \emph{Derived categories of twisted sheaves on elliptic threefolds}, J. Reine Angew. Math. \textbf{544} (2002), 161--179.

\bibitem{CanonacoStellari}  
A.~Canonaco \& P.~Stellari, \emph{Twisted Fourier--Mukai functors}, Adv. Math. 
\textbf{212}(2007), no. 2, 484--503.


\bibitem{CurtisReiner}  Ch.~Curtis \& I.~Reiner, \emph{Methods of representation theory. Vol. I. With applications to finite groups and orders},  Pure and Applied Mathematics,  A Wiley--Interscience Publication,  1981.

\bibitem{Deuring}
M.~Deuring, \emph{Die Anzahl der Typen von Maximalordnungen in einer Quaternionenalgebra von primer Grundzahl}, Nachr. Akad. Wiss. G\"ottingen. Math.-Phys. Kl. Math.-Phys.-Chem. Abt.  (1945) 48--50. 

\bibitem{Dieudonne}
J.~Dieudonn\'e, \emph{Topics in Local Algebra},  Notre Dame: University of Notre Dame Press, 1967.

\bibitem{DrozdTurchin}
 Yu.~Drozd \& V.~Turchin,  \emph{Locally conjugate orders}, 
 Mat. Zametki \textbf{24} (1978), no. 6, 871--878, 895. 


\bibitem{Gabriel}
P.~Gabriel, \emph{Des cat\'egories  ab\'eliennes},
{Bull. Soc. Math. France} \textbf{90} (1962), 323--448.

\bibitem{GeigleLenzing}
W.~Geigle \& H.~Lenzing,
A class of weighted projective curves arising in representation theory
of finite-dimensional algebras,
 {\it Singularities, representation of algebras, and vector bundles}, 265--297,
Lecture Notes in Mathematics \textbf{1273}, Springer   (1987).

\bibitem{GilleSzamuely}
Ph.~Gille \& T.~Szamuely,  \emph{Central Simple Algebras and Galois Cohomology},
 Cambridge Studies in Advanced Mathematics \textbf{101},  Cambridge University Press (2006).

\bibitem{GoodearlWarfield}
 K.~Goodearl \& R.~Warfield, \emph{An introduction to noncommutative Noetherian rings}, London Mathematical Society Student Texts \textbf{61},  Cambridge University Press (2004).


\bibitem{EGAII}
A.~Grothendieck, \emph{\'El\'ements de g\'eom\'etrie alg\`ebrique II. \'Etude globale \'el\'ementaire de quelques classes de morphismes}, Inst. Hautes \'Etudes Sci. Publ. Math.  \textbf{8} (1961), 222 pp.

\bibitem{Harada}
 M.~Harada, \emph{Hereditary orders},  Trans. Amer. Math. Soc. \textbf{107} (1963) 273--290.

\bibitem{ResiduesDuality}
R.~Hartshorne, \emph{
Residues and duality}, Lecture Notes in Mathematics \textbf{20},  Springer, 
Berlin--New York (1966).

\bibitem{Kussin}
 D.~Kussin, \emph{Weighted noncommutative regular projective curves},  J. Noncommut. Geom. \textbf{10} (2016), no. 4, 1465--1540.


\bibitem{Lam}
T.~Lam, 
\emph{Lectures on modules and rings},
Graduate Texts in Mathematics \textbf{189}, Springer--Verlag, New York, 1999.


\bibitem{LamFirstCourse}
T.~Lam, \emph{A first course in noncommutative rings}, 
 Second edition. Graduate Texts in Mathematics \textbf{131}. Springer--Verlag, New York, 2001.

\bibitem{LaumonRapoportStuhler} 
  G.~Laumon, M.~Rapoport \& U.~Stuhler, \emph{D--elliptic sheaves and the Langlands correspondence},  Invent. Math. \textbf{113} (1993), no. 2, 217--338.
  
\bibitem{LekiliPolishchuk}
Y.~Lekili \& A.~Polishchuk, \emph{Auslander orders over nodal stacky curves and partially wrapped Fukaya categories},   J. Topol. \textbf{11}  (2018), no. 3, 
615--644.

\bibitem{LenzingReiten} 
H.~Lenzing \& I.~Reiten, \emph{Hereditary Noetherian categories of positive Euler characteristic},  Math. Z. \textbf{254} (2006), no. 1, 133--171.

\bibitem{Liu} 
  Q.~Liu, \emph{Algebraic geometry and arithmetic curves},   Oxford Graduate Texts in Mathematics \textbf{6}. Oxford University Press, Oxford (2002). 
 
 \bibitem{Matsumura} H.~Matsumura, \emph{Commutative ring theory}, 
 Cambridge Studies in Advanced Mathematics \textbf{8}, 
  Cambridge University Press, Cambridge (1989). 

\bibitem{Matlis} 
  E.~Matlis, \emph{Injective modules over Noetherian rings},  Pacific J. Math. 
  \textbf{8} (1958), 511--528.

\bibitem{Milne}  
J.~Milne,  \emph{\'Etale cohomology}, 
Princeton Mathematical Series \textbf{33},  Princeton University Press (1980). 

\bibitem{Perego}
A.~Perego, \emph{A Gabriel theorem for coherent twisted sheaves}, Math. Z. \textbf{262} (2009), no. 3, 571--583.

\bibitem{Popescu} N.~Popescu, \emph{Abelian categories with applications to rings and
modules}, London Mathematical Society Monographs, no. \textbf{3},
 Academic Press  (1973).



\bibitem{ReinerMO}
I.~Reiner, \emph{Maximal orders}, London Mathematical Society Monographs,  New Series \textbf{28}. The Clarendon Press, Oxford University Press, Oxford, 2003.

\bibitem{ReitenvandenBergh}
I.~Reiten \& M.~Van den Bergh, \emph{Noetherian hereditary abelian categories satisfying Serre duality},  J.~Amer.~Math.~Soc.~\textbf{15} (2002), no. 2, 295--366.

\bibitem{Spiess}
 M.~Spie\ss{}, \emph{Twists of Drinfeld--Stuhler modular varieties},  Doc. Math. 2010, Extra vol.: Andrei A. Suslin sixtieth birthday, 595--654.

\bibitem{Toen}
B.~To\"en,  \emph{Derived Azumaya algebras and generators for twisted derived categories}, 
Invent. Math. \textbf{189} (2012), no. 3, 581--652.



\end{thebibliography}
\end{document}